\documentclass[11pt,letterpaper,reqno]{amsart}
\usepackage[margin=1in,letterpaper]{geometry}
\usepackage{graphicx}
\usepackage{amssymb}
\usepackage{amsthm}
\usepackage{mathrsfs}
\usepackage{stmaryrd}
\usepackage{accents} 
\usepackage{enumitem} 
\usepackage{subcaption}
\usepackage{microtype}
\usepackage{colonequals}
\DeclareMathAlphabet{\mathpzc}{OT1}{pzc}{m}{it}

\usepackage[bookmarksopen,bookmarksdepth=2]{hyperref} 
\allowdisplaybreaks

\makeatletter\let\over\@@over\makeatother

\numberwithin{equation}{section}
\theoremstyle{plain} 
\newtheorem{theorem}{Theorem}[section] 
\newtheorem{proposition}[theorem]{Proposition} 
\newtheorem{corollary}[theorem]{Corollary}
\newtheorem{lemma}[theorem]{Lemma} 
\theoremstyle{remark}
\newtheorem{remark}[theorem]{Remark}
\theoremstyle{definition}

\newcommand{\be}{\begin{equation}}
\newcommand{\ee}{\end{equation}}%
\newcommand{\bse}{\begin{subequations}}
\newcommand{\ese}{\end{subequations}}

\newcommand{\dist}{\operatorname{dist}}
\newcommand{\realpart}{\operatorname{Re}}
\newcommand{\imagpart}{\operatorname{Im}}

\newcommand{\range}{\operatorname{rng}}

\newcommand{\lspan}{\operatorname{span}}

\newcommand{\codim}{\operatorname{codim}}

\newcommand{\id}{\operatorname{id}}

\newcommand{\placeholder}{\,\cdot\,}

\newcommand{\n}[2][]{#1\lVert #2 #1\rVert}

\newcommand{\dell}{\partial}

\newcommand{\loc}{{\mathrm{loc}} }

\newcommand\A{\mathscr A}    
\newcommand\B{\mathscr B}    
\newcommand\F{\mathscr F}    

\newcommand\Nconf{N_{\mathrm{c}}}
\newcommand\Nvel{N_{\mathrm{v}}}

\newcommand\Wspace{\mathscr W}
\newcommand\Xspace{\mathscr X}
\newcommand\Yspace{\mathscr Y}

\newcommand{\cm}{{\mathscr C}}

\newcommand\fluidD{\mathscr{D}}

\newcommand{\confD}{\mathcal{D}}

\newcommand{\fullparam}{\Lambda}
\newcommand{\param}{\lambda}
\newcommand{\proj}{P}

\newcommand{\imagimag}{{\mathrm{ii}} }        
\newcommand{\realimag}{{\mathrm{ri}} }

\begin{document}

\title[Periodic hollow vortices]{Global bifurcation of hollow vortex streets}

\date{\today}

\author[V.N. Oikonomou]{Vasileios N. Oikonomou}

\address{Department of Mathematics, University of Missouri, Columbia, MO 65211}
\email{oikonomou@missouri.edu}
\author[S. Walsh]{Samuel Walsh,}
\address{Department of Mathematics, University of Missouri, Columbia, MO 65211} 
\email{walshsa@missouri.edu} 

\begin{abstract}
Vortex streets are periodic configurations of vortices propagating through an irrotational flow. In this paper, we study streets of hollow vortices, which are solutions to the free boundary $2$-d irrotational incompressible Euler equations. Each vortex core is a region of constant pressure in the complement of the fluid domain with a nonzero circulation around it.  We prove that any non-degenerate singly-periodic point vortex configuration can be ``desingularized'' to create a global curve of solutions to the steady hollow vortex street problem, and we further characterize the types of singular behavior that can develop as one transverses the curve to its extreme. As specific examples, we study von K\'arm\'an vortex streets, translating vortex arrays, and a two-pair (2P) configuration.  Our method is based on analytic global bifurcation theory and adapts the desingularization technique of Chen, Walsh, and Wheeler~\cite{chen2023desingularization} to the periodic setting.
\end{abstract}

\maketitle

\setcounter{tocdepth}{1}
\tableofcontents

\section{Introduction}
\label{introduction section}

Persistent regions of concentrated vorticity are an important feature of two-dimensional incompressible inviscid motion. Perhaps the most famous example are the \emph{vortex streets} that can develop in the turbulent wake of a blunt body. These take the form of a periodic procession of compact vorticity regions translating at fixed velocity through a background irrotational flow. Similar structures are found in myriad other physical contexts. 

There are several ways to model this behavior mathematically. The simplest is to consider configurations of point vortices governed by the well-known Kirchhoff--Helmholtz equation, which is a finite-dimensional Hamiltonian system. These, however, are too singular to be solutions of the Euler equations even in the weak sense. In this paper, we study \emph{hollow vortices}, which are another classical type of localized vorticity flow.  As we describe in more detail below, they are solutions of the full free boundary incompressible Euler system. One treats the vortex cores as bounded regions of constant pressure in the complement of the fluid domain. The flow in the bulk of the fluid is assumed to be irrotational, but with a nonzero circulation about each core. In that sense, the vorticity is confined to the boundary, which is the union of Jordan curves. This should be contrasted with vortex patches, which are solutions of the incompressible Euler equations in the plane for which the vorticity is supported on a disjoint union of bounded open sets. 

Hollow vortices have been investigated since the 19th century~\cite{pocklington1895} and have recently enjoyed considerable renewed interest. Explicit linear arrays of hollow vortices were found by Baker, Saffman, and Sheffield~\cite{baker1976structure}.  A variety of steady hollow vortex configurations, both periodic and not, have been constructed using tools from special function theory~\cite{Crowdy_Nelson_Krishnamurthy_2021, Crowdy_2014Hollow, crowdy2011analytical, DGLShollowstreet,chen2025finite,chen2025vortex, nelson2021corotating, CrowdyKrish2018Compressible}. Many of these solutions are in fact explicit. Traizet~\cite{traizet2015hollow}, on the other hand, showed that there is a correspondence between certain steady hollow vortex configurations and minimal surfaces. 

In this paper, we instead take a bifurcation theoretic approach, adapting the vortex desingularization and global continuation method introduced in~\cite{chen2023desingularization,chen2025finite} to the periodic case. This entails many nontrivial changes to the general framework in both the analytical and geometric settings. Our main result states essentially that a generic periodic steady point vortex configuration gives rise to a large family of periodic hollow vortex configurations. As just two applications, we give a new construction of a family of von K\'arm\'an hollow vortex streets and the first existence result for non-perturbative configurations of $2$P of hollow vortices. Both of these connects from point vortices on one end to singular solutions of the free boundary Euler system on the other.

\begin{figure}
\centering
\includegraphics[scale=0.9]{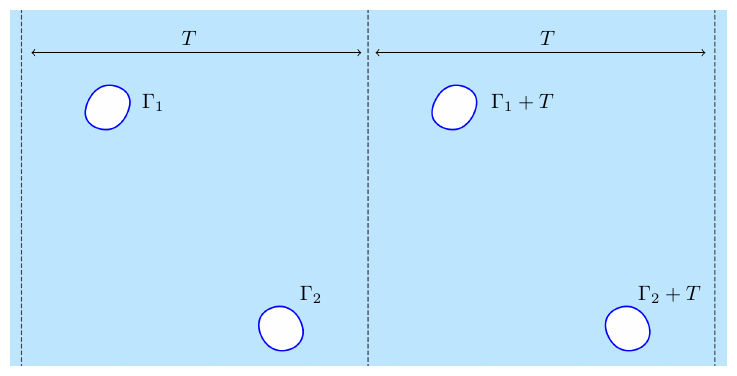}
\caption{An illustration of the physical domain $\mathscr{D}$ in a case that $M=2$. In this case, we have two hollow vortices per period}
\label{physicaldomain}
\end{figure}

\subsection{Governing equations for hollow vortices} \label{governingequations}

Consider an ideal fluid lying in a (time-dependent) planar domain $\mathscr{D}(t)$. The fluid velocity field $\mathbf{u}(t) \colon \mathscr{D}(t) \to \mathbb{R}^2$ is irrotational and obeys the two-dimensional incompressible, Euler equations:  
\begin{subequations} \label{time-dependent Euler equations}
\begin{equation}\label{EulerequationsPDE}
\begin{cases} 
	\partial_t \mathbf{u} + (\mathbf{u}\cdot \nabla) \mathbf{u} =- \nabla P & \text{ in } \mathscr{D}(t) \\
\nabla^{\perp}\cdot \mathbf{u}=0 &\text{ in } \mathscr{D}(t) \\
\nabla \cdot \mathbf{u}=0 &\text{ in } \mathscr{D}(t),
\end{cases}
\end{equation}
where $P(t) \colon \fluidD(t) \to \mathbb{R}$ is the pressure.  Motivated by the vortex street example, we suppose that $\mathbf{u}(t)$, $P(t)$, and $\fluidD(t)$ are $T$-periodic in the horizontal direction; denote by $\fluidD^0(t)$ the restriction of $\fluidD(t)$ to the fundamental periodic strip containing the origin. 

For the hollow vortex model,  $\fluidD(t)$ is an exterior domain with each connected component of its complement representing a single vortex. We assume that there are $1 \leq M < \infty$ vortices per period, and thus $\partial \fluidD^0(t)$ consists of a disjoint union of Jordan curves $\Gamma_1(t), \ldots, \Gamma_M(t)$ ignoring the left and right vertical lines; see Figure~\ref{physicaldomain}). By convention, we always treat $\Gamma_k(t)$ as oriented counter-clockwise.

The \emph{kinematic condition} requires the fluid particles on the boundary stay on the boundary, or equivalently that
\begin{equation}
	\label{time-dependent kinematic condition}
	\mathbf{u} \cdot \mathbf{n} = \mathfrak{v}_k \qquad \textrm{on }  \Gamma_k(t) \qquad \textrm{for } k = 1, \ldots, M,
\end{equation}
where $\mathfrak{v}_k$ is the normal velocity of the $k$-th vortex boundary and $\mathbf{n}$ is the unit outward normal. The \emph{dynamic condition} states that the pressure is continuous over the interface. As each vortex core is taken to be a region of constant (in space) pressure, this becomes
\begin{equation}
	\label{time-dependent dynamic condition}
	P = P_k \quad \textrm{on } \Gamma_k(t) \qquad \textrm{for } k = 1, \ldots, M.
\end{equation}
Here, $P_k = P_k(t)$ is the pressure in the core bounded by $\Gamma_k(t)$. Finally, we assume that the velocity field has a nonzero circulation about each vortex core:
\begin{equation}
	\label{time-dependent circulation condition}
	\int_{\Gamma_k(t)} \mathbf{u}(t) \cdot \mathbf{t} \, ds = \gamma_k \qquad \textrm{for } k = 1, \ldots, M,
\end{equation}
with $\mathbf{t}$ being the unit tangent vector along $\Gamma_k(t)$. The constants (in space and time) $\gamma_1, \ldots, \gamma_M$ physically represent the strength of the corresponding hollow vortices.
\end{subequations}

Structurally, one should think of the hollow vortex system~\eqref{time-dependent Euler equations} as analogous to the classical water wave problem but with a different geometry: rather than having an incompressible fluid region beneath a constant pressure air region, here the fluid inhabits an exterior domain surrounding disjoint constant pressure regions.  In the water waves setting, local well-posedness hinges on the well-known Rayleigh--Taylor sign condition, which requires  
\begin{equation}
  \label{taylor sign condition}
  -\mathbf{n} \cdot \nabla P \ge \kappa > 0 \qquad \textrm{on } \partial \fluidD(t)
\end{equation}
for some positive constant $\kappa$. So long as the sign condition~\eqref{taylor sign condition} holds, one should expect the Cauchy problem for the hollow vortex system~\eqref{time-dependent Euler equations} is likewise (locally) well-posed in sufficiently high regularity Sobolev spaces as a consequence of the analysis in~\cite{shatah2011interface}, for example.

\subsection{Steady hollow vortices}

Our interest here is not in the general time-dependent problem~\eqref{time-dependent Euler equations}, but rather in the existence of special \emph{steady} configurations that are independent of time when viewed in a moving reference frame that is translating with some constant velocity $c \in \mathbb{R}^2$. This situation describes the behavior of a fully developed vortex street far downstream from a blunt body, for instance. Denote by $\fluidD$ the fluid domain in the moving frame, and define $\fluidD^0$, and $\Gamma_1, \ldots, \Gamma_M$ accordingly. Abusing notation, we will from now on write $\mathbf{u}$ and $P$ for the velocity field and pressure in the moving frame, which are hence time-independent functions with domain $\overline{\fluidD}$. Naturally, the core pressures $p_1, \ldots, p_M$ are also taken to be constant.

\begin{subequations}\label{PHVP}

The steady problem can be reformulated quite elegantly using tools from complex analysis. A point $(x,y)\in \mathbb{R}^2$ in the physical domain $\mathscr{D}$ is associated in the usual way with the complex number $z = x+i y\in \mathbb{C}$.  The requirement in~\eqref{EulerequationsPDE} that $\mathbf{u} = (u, v)$ is divergence free and curl free is equivalent to the statement that the complex velocity $\mathsf{u} \colonequals u - iv$ is holomorphic in $\fluidD$. We may then introduce a complex potential
\begin{equation}\label{holomorphicityofvelocity}
\mathsf{w} \colon \mathscr{D}\to \mathbb{C} \quad \text{holomorphic and $T$-periodic}
\end{equation}
such that $\partial_z \mathsf{w} = \mathsf{u}$. Note that $\mathsf{w}$ is necessarily multi-valued, and, by~\eqref{time-dependent circulation condition}, it satisfies
\begin{equation}\label{circulationconditionintro}
\int_{\Gamma_k} \partial_z \mathsf{w} \, dz = \gamma_k \qquad \textrm{for } k = 1, \ldots, M.
\end{equation}
The relative velocity field likewise has a relative complex potential $\mathsf{w}-cz$. One can show that the {kinematic condition}~\eqref{time-dependent kinematic condition} is equivalent to the vortex boundaries being relative streamlines, which stated in terms of the relative potential amounts to requiring that
\begin{equation}\label{kinematicconditionintro}
\imagpart{\left(\mathsf{w}-cz\right)} = m_k \quad \text{on } \Gamma_k \qquad \textrm{for } k = 1, \ldots, M, 
\end{equation}
for some constant fluxes $m_1, \ldots, m_M \in \mathbb{R}$. Lastly, it is convenient to restate the dynamic condition~\eqref{time-dependent dynamic condition} in terms of the potential rather than the pressure. Integrating the steady Euler equations gives Bernoulli's law, which in the irrotational and incompressible setting takes the form
\[
	\frac{1}{2}|\partial_z \mathsf{w}-c|^2  + P = C \quad \text{ in } \mathscr{D},
\]
for a constant $C$. Evaluating this on each vortex boundary, the dynamic condition~\eqref{time-dependent dynamic condition} becomes 
\begin{equation}\label{bernoulliconditionintro}
	|\partial_z \mathsf{w}-c| = q_k \quad \text{on } \Gamma_k \qquad \textrm{for } k = 1, \ldots, M,
\end{equation}
where $q_1, \ldots, q_M \in \mathbb{R}$ are the so-called Bernoulli constants.
\end{subequations}

\subsection{Steady point vortices}
\label{point vortex section}

Our strategy will be to construct solutions of the steady hollow vortex problem~\eqref{PHVP} by desingularizing steady point vortex configurations. Let us begin by briefly recalling the Kirchhoff--Helmholtz model in the periodic setting. For more thorough discussion, see, for example, \cite{aref2007playground, StemlerBasu2Pwakes, 2Pflappingfoil, Stremler20142P, Stemler2P2C, ArefExoticvortexwakes, Stemler2003equilibria}.

The basic premise of the point vortex model is that the vorticity $\omega \colonequals \nabla^\perp \cdot \mathbf{u}$ is the sum of Dirac $\delta$ measures:
\begin{equation*}
    \omega = \sum_{n \in \mathbb{Z}} \sum_{k=1}^M \gamma_k \delta_{z_k+ nT}.
\end{equation*}
Here, $z_1, \ldots, z_M$ are the (finitely many) vortex centers lying in a fundamental periodic strip and  $\gamma_1, \ldots, \gamma_M$ are the corresponding vortex strengths. The complex velocity field would then take the form
\begin{equation}
\label{point vortex complex velocity field}
	u - i v = \sum_{n \in \mathbb{Z}} \sum_{k = 1}^M \frac{\gamma_k}{2\pi i} \frac{1}{ \placeholder - z_k - nT}.
\end{equation}
Since in the two-dimensional Euler equations, the vorticity is transported, we expect that each center $z_k$ evolves according to 
\[
 	\partial_t\overline{z_k} = \sum_{n \in \mathbb{Z}} \sum_{\substack{j=1 \\ j\neq k}}^{M} \frac{\gamma_j}{2\pi i }\frac{1}{z_k-z_j - n T} \qquad \text{for } k=1,\ldots , M.
\]
Note that we have eliminated the (singular) self-advection term.  The above series is not absolutely convergent, but formally rearranging it we arrive at the classical Helmholtz--Kirchhoff model for the $T$-periodic case:
\begin{equation}\label{periodic Helmholtz--Kirchhoff equation}
   \partial_t \overline{z_k}= \sum_{\substack{j=1 \\ j\neq k}}^M \frac{\gamma_j}{2iT}\cot{\left (\frac{\pi (z_k-z_j)}{T} \right)} \qquad \textrm{for } k = 1, \ldots, M.
\end{equation}
This calculation uses the well-known expansion formula for the cotangent:
\begin{equation}\label{cotexp}
\pi \cot(\pi z ) = \frac{1}{z} + \sum_{k=1}^{\infty}\frac{2z}{z^2-k^2},
\end{equation}
which can be found, for example, in \cite{aref2007playground, ConwayComplex}. Observe now that the system~\eqref{periodic Helmholtz--Kirchhoff equation} is finite-dimensional as it suffices to understand the positions of the $M$ vortices $z_1, \ldots, z_M$. We mention that a similar formal argument leads to the periodic complex potential
\begin{equation}\label{velocityforpointvortexproblem}
    w_0 = \sum_{k=1}^{M} \frac{\gamma_k}{2\pi i } \log{\sin{\left(\frac{\pi }{T} (\placeholder-z_k) \right)}}.
\end{equation}
The corresponding complex velocity field will then have simple poles at $z = z_k+T\mathbb{Z}$, which is consistent with~\eqref{point vortex complex velocity field}.

We say that a configuration of periodic point vortices with vortex centers $z_1,\ldots,z_M $ and respective circulations $\gamma_1,\ldots ,\gamma_M$ is \emph{steady} if $\partial_t \overline{z_k}  =c $ for all $k=1,\ldots ,M $ for some $c \in \mathbb{C}$; if $c \neq 0$, it is said to be \emph{traveling} and it is \emph{stationary} if $c=0$. When there are finitely many point vortices, it is possible to have steady configurations that rotate with constant angular velocity, however these do not exist in the periodic setting~\cite{Stemler2003equilibria}. 

We gather the parameters for the periodic steady point vortex problem into the vector 
\begin{equation*}
	\Lambda = (\zeta_1,\ldots ,\zeta_M, \gamma_1,\ldots ,\gamma_M,c,T) \in \mathbb{P}  \colonequals \mathbb{C}^M\times \mathbb{R}^M\times \mathbb{C}\times \mathbb{R}.
\end{equation*}
Here, we are writing $\zeta_1, \ldots, \zeta_M$ for the vortex centers (which are now fixed). In light of the time-dependent problem~\eqref{periodic Helmholtz--Kirchhoff equation}, the parameters must satisfy the algebraic system of equations
\be
\label{zero-set mathcal V}
	\mathcal{V}(\Lambda) = 0,
\ee
Here, we are writing $\zeta_1, \ldots, \zeta_M$ for the vortex centers (which are now fixed). In light of the time-dependent problem~\eqref{periodic Helmholtz--Kirchhoff equation}, the parameters must satisfy the algebraic system of equations
\be
\label{zero-set mathcal V}
	\mathcal{V}(\Lambda) = 0,
\ee
where the nonlinear map $\mathbb{P}\ni \Lambda\mapsto \mathcal{V}(\Lambda) = (\mathcal{V}_1,\ldots , \mathcal{V}_M)\in \mathbb{C}^M$ 
has components given by
\be
\label{definition V_k}
\mathcal{V}_k (\Lambda) \colonequals \sum_{j\neq k}\frac{\gamma_j}{2iT}\cot{\left(\frac{\pi (\zeta_k-\zeta_j)}{T} \right)} -c \qquad \textrm{for } k = 1, \ldots, M.
\ee
Many explicit solutions are known for this system, including idealizations of vortex streets. It is somewhat remarkable that despite the many formal manipulations needed to obtain~\eqref{periodic Helmholtz--Kirchhoff equation}, it is nonetheless a highly successful model of concentrated vorticity flow.

Finally, we say $\Lambda_0$ is a \emph{non-degenerate configuration} provided $\mathcal{V}(\Lambda_0) = 0$ and $D_{\Lambda}\mathcal{V}(\Lambda_0)$ has full rank. In this case, we may decompose the parameter space $\mathbb{P} = (\lambda,\lambda')$ so that $D_{\lambda} \mathcal{V}(\Lambda_0)$ is an isomorphism. We mention that this is slightly different from the definition for the finite vortex case~\cite{chen2023desingularization} because here, we are requiring that the configurations are periodic in the direction of the real axis while letting the wave speed $c$ be complex valued.  

\subsection{Conformal variable formulation of the hollow vortex problem}
\label{intro conformal section}

Recall that the unknowns for the steady hollow vortex problem~\eqref{PHVP} are the fluid domain $\fluidD$ and complex potential $\mathsf{w}$. The basic approach is to take a given non-degenerate periodic steady point vortex configuration and seek a periodic steady hollow vortex configuration that is nearby in the sense that the vortex boundaries are close to streamlines of the point vortex system. Of course, one of the main analytical difficulties is that this is a free boundary problem. With that in mind, we construct the fluid domain as the image under an a priori unknown conformal mapping of a fixed \emph{canonical domain}. The convention will be that the conformal domain is a subset of the complex $\zeta$-plane, and we will continue to use $z$ for the (complexified) physical variables.   From~\eqref{velocityforpointvortexproblem}, it follows that the streamlines for the point vortex system are asymptotically circular in the neighborhood of each vortex, and so a reasonable choice is to take
\be
\label{definition conformal domain}
	\confD_\rho = \mathcal{D}_{\rho} (\Lambda) \colonequals \mathbb{C}\setminus \left( \bigcup_{k=1}^{M} \overline{B_{|\rho|}(\zeta_k)}  + T\mathbb{Z} \right) \qquad \textrm{for } 0 < |\rho| \ll 1.
\ee
Here, $|\rho| > 0$ is the conformal radius, and we abuse notation slightly by writing $\confD_0 \colonequals \mathbb{C} \setminus \{ \zeta_1 + T\mathbb{Z}, \ldots, \zeta_M + T\mathbb{Z}\}$, which is the domain of definition of the starting point vortex velocity field. Naturally, we always assume $\rho$ is sufficiently small that the boundary of $\confD_\rho$ has precisely $M$ connected components. 

We may then introduce as a new unknown a injective and conformal mapping $f \colon \mathcal{D}_{\rho}\to \mathbb{C}$, and define the fluid domain in the $z$-plane by $\mathscr{D}\colonequals f(\mathcal{D}_{\rho})$; see Figure~\ref{pointtocirculartojordan}. Note that for $0 < |\rho| \ll 1$, we expect that $f$ will be near identity.  We denote by $w$ the complex potential in the conformal domain $\confD_\rho$. For $0 < |\rho| \ll 1$,  our ansatz will treat it as a perturbation of the complex potential $w_0$ for the point vortex system~\eqref{velocityforpointvortexproblem}. The entire problem can then be reformulated in the $\zeta$-plane for the unknowns
\be
\label{Holder regularity f w}
f \in C^{\ell+\alpha}(\overline{\confD_\rho}), \qquad w \in C^{\ell+\alpha}(\overline{\confD_\rho}),
\ee
where $\ell \geq 2$ and $\alpha \in (0,1)$ are arbitrary but fixed. Once $f$ and $w$ are determined, the physical domain is $\fluidD = f(\confD_\rho)$, and the complex velocity in the physical domain is given by $(w_\zeta/ f_\zeta) \circ  f^{-1}$.

\begin{figure}
	\includegraphics[width=0.9\linewidth]{conforrmal_setup-wavy.pdf}
	\caption{Left: fundamental periodic strip for a starting configuration of four point vortices. Center: fundamental strip for the conformal domain $\confD_\rho$ with $0 < |\rho|$. Right: one period of the corresponding physical domain $\fluidD = f(\confD_\rho)$. Each vortex boundary $\Gamma_k = f(\partial B_\rho(\zeta_k))$.}
	\label{pointtocirculartojordan}
\end{figure}

Finally, it will be necessary to work with a normalized version of the Bernoulli constants. Note that from~\eqref{velocityforpointvortexproblem} it follows that for a steady configuration
\be
	\label{point vortex relative velocity asymptotics}
	\partial_\zeta w_0(\zeta) = \frac{\gamma_k}{2\pi } \frac{1}{\zeta-\zeta_k} + c+ O(\zeta-\zeta_k) \qquad \textrm{as } \zeta \to \zeta_k.
\ee
 The relative velocity field for the hollow vortex system in the conformal variables is given by
\be
\label{definition U}
	U\colonequals \partial_z \left( w- c f \right) = \frac{\partial_\zeta w}{\partial_\zeta f} -c.
\ee
The dynamic condition~\eqref{bernoulliconditionintro} requires that $U$ has constant modulus $q_k$ on the $k$-th vortex boundary. In light of the asymptotics~\eqref{point vortex relative velocity asymptotics}, though, we expect $q_k = O(1/\rho)$ as $\rho \searrow 0$. We therefore define normalized Bernoulli constants $Q_1, \ldots, Q_M$ via
$$
\rho^2q_k^2 \equalscolon \frac{\gamma_k^2}{4\pi^2} + \rho Q_k.
$$
The precise scaling is justified by expanding the dynamic condition on $\partial B_\rho(\zeta_k)$ in powers of $\rho$; see Section~\ref{bernoulli section}.

\subsection{Statement of results} \label{statement section}

We now present the main contributions of this paper. For the time being, this will be done somewhat informally, as the actual analysis will be done on a transformed problem with fixed domain. 

Our first result states that generically, periodic point vortex configurations can be ``desingularized'' to create families of periodic hollow vortex configurations.

\begin{theorem}[Desingularization] \label{local desingularization theorem}
Fix any $\ell \geq 2$ and $\alpha \in (0,1)$. Let $\Lambda_0= (\lambda,\lambda')$ be a non-degenerate steady periodic point vortex configuration. Then there exists a curve $\cm_\loc$ of steady periodic hollow vortex configurations with the regularity~\eqref{Holder regularity f w}. The curve admits the real-analytic parameterization
\begin{equation*}
    \mathscr{C}_{\textup{loc}} = \left\{ \left( f({\rho}),w({\rho}),Q({\rho}), \lambda({\rho}), \rho \right) :  |\rho| \ll 1 \right\},
\end{equation*}
 where
 \[
 	f(0) = \id, \qquad w(0) = w_0, \qquad Q(0) = 0, \qquad \lambda(0) = \lambda_0.
 \]
\end{theorem}

This theorem generalizes the vortex desingularization result in \cite{chen2023desingularization} to the periodic setting. It is worth noting that the definition of non-degenerate is slightly different in the present paper. Theorem~\ref{local desingularization theorem} is proved using the analytic implicit function theorem. While it does not give explicit solutions in terms of special functions, one can (rigorously) expand to arbitrary order in $\rho$. There are many similar results for the existence of perturbative steady vortex patch solutions as desingularized point vortices; see, for example~\cite{wan1988desingularizations,cao2014regularization,hmidi2017pairs,hassainia2022multipole}.

Our second result extends the $\mathscr{C}_{\textup{loc}}$ to a maximal locally real-analytic and globally $C^0$ curve using techniques from global bifurcation theory. Specifically, as we follow the global curve to its extreme, we show that one of two types of singular behavior occurs. The first, which following~\cite{chen2023desingularization}, we call \emph{conformal degeneracy} is captured by the blowup of the quantity
\begin{equation}\label{Ncdefinition}
    N_{\textup{c}} (f) \colonequals \sup_{\partial \confD_{\rho}}|\partial_{\zeta}f| + \sup_{\zeta,\zeta'\in \partial \mathcal{D}_{\rho}} \frac{|\zeta-\zeta'|}{|f(\zeta)-f(\zeta')|}+ \frac{1}{\displaystyle\min_{j\neq k}\text{dist}(\Gamma_j,\Gamma_k)} +\frac{1}{\displaystyle\min_{j,k}\text{dist}(\Gamma_j+T,\Gamma_k)}.
\end{equation}
Observe that blowup of the second term indicates either loss of conformality and/or loss of injectivity of $f$ on the boundary. The latter of these happens, for example, when a vortex self-intersects. Blowup of the last two terms indicates that two of the vortex boundaries in the conformal domain are coming into contact. A new wrinkle in the periodic setting is that this may occur by a vortex from one periodic strip coming into contact with those in an adjacent strip. 

The second alternative, which we call \emph{velocity degeneracy}, is that the relative velocity on one of the vortex boundaries tends to $0$ or blows up in magnitude. This is described by the blowup of the quantity
\begin{equation}\label{Nvdefinition}
    N_{\textup{v}} (f,w)\colonequals \sup_{\partial \mathcal{D}_{\rho}} \left ( |U| + \frac{1}{|U|} \right).
\end{equation}
Recall that $U$ is defined by~\eqref{definition U}.

With that terminology in mind, the second theorem is then as follows.

\begin{figure}
    \centering
    \includegraphics[]{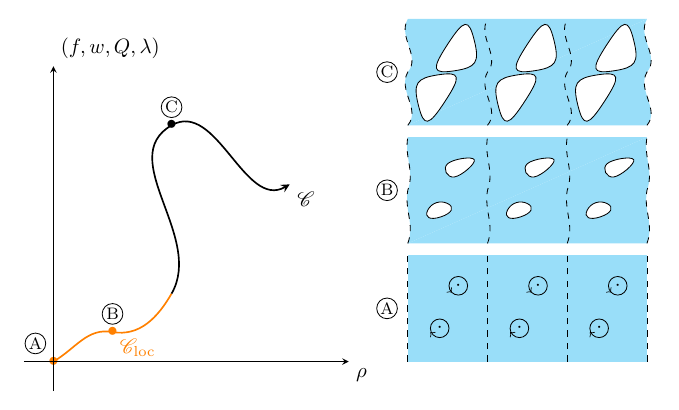}
    \caption{A schematic illustration of the global bifurcation curve $\cm$ given by Theorem~\ref{global desingularization theorem}. It bifurcates from an initial (non-degenerate) point vortex figuration and contains the local curve $\cm_\loc$ given by Theorem~\ref{local desingularization theorem}. While $\cm_\loc$ is parameterized by the conformal radius $\rho$, the global curve may have multiple turning points. It is also possible that secondary bifurcations occur, which are not pictured. As we follow $\cm$ to its extreme, either $|\lambda(s)|$ is unbounded or else conformal degeneracy or velocity degeneracy occurs.}
    \label{bifurcationdiagram}
\end{figure}

\begin{theorem}[Global continuation] \label{global desingularization theorem}
Let $\Lambda_0= (\lambda_0, \lambda_0')$ be a non-degenerate periodic point vortex configuration. Assume that at least one of the circulations is in $\lambda_0'$. Let $\cm_\loc$ be the curve of hollow vortex solutions bifurcating from $\Lambda_0$ given by Theorem~\ref{local desingularization theorem}. There exists a curve $\cm$ of hollow vortex solutions with $\cm_\loc \subset \cm$ and such that the following statements hold. 
\begin{enumerate}[label=\rm(\alph*)]
	\item The curve $\cm$ admits the global $C^0$ parameterization 
\begin{equation*}
    \cm = \{ (f(s),w(s),Q(s), \lambda(s), \rho(s)) :  s\in [0,\infty) \}
\end{equation*}
	with 
	\[
		f(0) = \id, \qquad w(0) = w_0 \qquad Q(0) = 0, \qquad \lambda(0) = \lambda_0, \qquad \rho(0) = 0.
	\]
	\item At each point, the curve $\cm$ has a local real-analytic reparameterization.  
	\item The quantity 
	\[
		N(s)\colonequals \Nconf(s) + \Nvel(s) + |\lambda(s)| \colonequals \Nconf(f(s)) + \Nvel(f(s), w(s)) + |\lambda(s)| 
	\]
	is finite for all $s \in (0,\infty)$, and
	\begin{equation}\label{blowupintro}
 		  \lim_{s\to \infty} \left( \Nconf(s) + \Nvel(s)+|\lambda(s)| \right) = \infty.
	\end{equation}
	\end{enumerate}
\end{theorem}

Theorem~\ref{global desingularization theorem} is obtained by using (analytic) global bifurcation theory. This is a very general tool which, when applied directly in this setting, allows for a far more expansive list of alternatives for the limiting behavior along the curve than what is stated in~\eqref{blowupintro}. One of the main strengths of our result is that the incipient degeneracy can be characterized rather economically in terms of the quantities $\Nconf$, $\Nvel$, and the parameters themselves, which all have physical significance. Similar vortex desingularization results have been obtained by a variety of analytical methods for vortex patches~\cite{garcia2020streets, GarciaHaziot2023, hassainia2020global}, as well as planar hollow vortices~\cite{chen2023desingularization}, and wave-borne hollow vortices~\cite{chen2025vortex}.

Next, we give three concrete applications of Theorems~\ref{local desingularization theorem} and \ref{global desingularization theorem}. First consider the (staggered) von K\'arm\'an point vortex configuration (with $M=2$ vortices per period):
$$
\Lambda_{\textup{vk}}^{0} \colonequals \left(\zeta_1=\tfrac{\pi}{4}+i,~\zeta_2=-\tfrac{\pi}{4}-i,~\gamma_1=1,~\gamma_2=-1,~ c= \tfrac{\tanh{(2)}}{2\pi},~ T= \pi \right).
$$
One can readily confirm that $\Lambda_{\textup{vk}}^0$ satisfies~\eqref{zero-set mathcal V}, and hence is a translating point vortex configuration.  

\begin{corollary}[von Kármán vortex street] \label{Von Karman vortex street initial Corollary}
There exists a family $\mathscr{C}_{\textup{vk}}$ of von K\'arm\'an vortex streets of hollow vortices admitting the global $C^0$ parametrization
$$
\mathscr{C}_{\textup{vk}} = \{ (f(s), w(s),Q(s), c(s) ,\rho(s)) : s \in [0,\infty) \}
$$
where $\zeta_1, \zeta_2, \gamma_1, \gamma_2$, and $T$ are fixed to their values in $\Lambda_{\textup{vk}}^0$.
\begin{enumerate}[label=\rm(\alph*)]
    \item The curve bifurcates from the point vortex configuration $\Lambda_{\textup{vk}}^0$ in that 
    \begin{equation}\label{VKintrocoordinates}
        f(0) = \id ,\quad w(0)=\frac{1}{2\pi i }\log{\left( \frac{\sin (\placeholder -\tfrac{\pi}{4}-i)}{\sin(\placeholder+\tfrac{\pi}{4}+i)} \right)},\quad Q(0)=0,\quad \rho(0)=0,\quad c(0)=\frac{\tanh(2)}{2\pi}.
    \end{equation}
    \item \label{vk street theorem blowup part}  For all $s>0$, we have $\rho(s)>0$. Moreover, in the limit 
    $$
    \lim_{s\to\infty} \left( \Nconf(s) + \Nvel(s) \right) = \infty.
    $$
    \item For each $s \geq 0$, $f(s)$ and $w(s)$ are odd in $\zeta$. Consequently,
    \be
    \label{von Karman Bernoulli constant symmetry}
    		-\mathscr{D}(s) =\mathscr{D}(s), \qquad Q_1(s) = Q_2(s) \qquad \textrm{for all } s \geq 0.
	\ee
\end{enumerate}
\end{corollary}
Note that by exploiting some symmetries in the configuration, all the parameters but the speed $c$ are fixed along the solution family. We are further able to provide an a priori bound on $c$ in terms of $\Nconf$ and $\Nvel$, leading to the more refined blowup statement in part~\ref{vk street theorem blowup part}. Families of von Kármán hollow vortex streets were previously computed by Crowdy and Green~\cite{crowdy2011analytical} using a representation via Schottky--Klein prime functions. In their numerics, the limiting vortex boundaries appear to have corners, which would imply blowup of $\Nconf$. Their family is normalized differently --- $c$ is fixed --- but it would be very interesting to see if a similar behavior occurs along $\cm_{\textrm{vk}}$.

Traizet~\cite{traizet2015hollow} rigorously constructed a one-parameter family of small, near-circular von Kármán hollow vortex streets. His approach relied on a remarkable correspondence between the steady hollow vortex system~\eqref{PHVP} and a minimal surface problem. Specifically, let $\widetilde{\fluidD} \subset \mathbb{R}^2$ be an unbounded domain that is periodic in the horizontal direction and has a nonempty boundary with finitely many components per period.  Consider classical solutions $v = v(x,y) \colon \widetilde{\fluidD} \to \mathbb{R}$ of the quasilinear elliptic problem
\be 
\label{minimal surface equation}
\left\{
	\begin{aligned}
		(1+ v_y^2 ) v_{xx} + ( 1+ v_x^2) v_{yy} - 2 v_x v_y v_{xy} & = 0 & \qquad & \textrm{in } \widetilde{\fluidD} \\
		v \textrm{ locally constant} & & \qquad  & \textrm{on } \partial \widetilde{\fluidD} \\
		| \nabla v | & \to \infty & \qquad & \textrm{on } \partial \widetilde{\fluidD}.
	\end{aligned}
\right.
\ee
The graph of $v$ is then a minimal surface. The blowup of $|\nabla v|$ as one approaches the boundary geometrically means that the normal to the surface (Gauss map) is purely horizontal there. Under the assumption that the Bernoulli constants are all the same and that $|U|$ is maximized on $\partial\fluidD$, Traizet proves that each solution of the steady hollow vortex problem~\eqref{PHVP} gives rise to a solution of the minimal surface problem~\eqref{minimal surface equation}, with $\widetilde{\fluidD}$ being the image of $\fluidD$ under an explicit mapping determined from $w$. Thanks to~\eqref{von Karman Bernoulli constant symmetry} and a maximum modulus principle argument, all of the solutions along $\cm_{\textrm{vk}}$ in fact satisfy both of those hypotheses, and thus there is a corresponding global family $\widetilde{\cm}_{\mathrm{vk}}$ of minimal surface solutions to \eqref{minimal surface equation}.

A typical example of a singly periodic vortex (unstaggered) array of point vortices (with $M=2$ vortices per period) is 
\begin{equation*}
    \Lambda_{\mathrm{ar}}^0 \colonequals \left( \zeta_1=i, \, \zeta_2=-i, \, \gamma_1=1,\, \gamma_2=-1,\, c=\tfrac{\coth(2)}{2\pi}, \, T = \pi \right).
\end{equation*}
Applying a version of Theorems~\ref{local desingularization theorem} and \ref{global desingularization theorem} furnishes a corresponding global family of periodic hollow vortex arrays $\mathscr{C}_{\textup{ar}}$, that bifurcates from the point vortex configuration
 \begin{equation}\label{ARintrocoordinates}
        f(0) = \id, \quad w(0) = \frac{1}{2\pi i } \log \Bigg( \frac{\sin (\cdot-i)}{{\sin(\cdot + i)}} \Bigg), \quad Q(0) =0, \quad c(0) = \frac{\coth (2)}{2\pi}.
\end{equation}
For a precise statement, and the symmetries exhibited by  $\mathscr{C}_{\textup{ar}}$ we refer to the statement of Corollary \ref{ArrayCorollaryfinal}. 

We also construct the first large families of hollow vortices $\mathscr{C}_{2\textup{P}}$ that bifurcate from steady (unstaggered) $2$P (two pair) point vortex configurations. A typical example is
\begin{align*}
      \Lambda_{\textup{2P}}^0 & \colonequals \Big( \zeta_1=i, \, \zeta_2=2i + \tfrac{\pi}{2},\, \zeta_3=-i,\,\zeta_4 = -2i + \tfrac{\pi}{2},\, \gamma_1 = \gamma_1^0, \\
      &\qquad\gamma_2=\gamma_2^0 ,\,\gamma_3 = -\gamma_1^0,\, \gamma_4=-\gamma_2^0,\, c=1,\,T=\pi\Big),
 \end{align*}
 where
 \begin{equation}\label{CUT1}
 	\gamma_1^0 \colonequals \tfrac{2\pi\left(-\coth{4}+\tanh{1}+\tanh{3}\right)}{-(\tanh{3})^2+(\tanh{1})^2+\coth{2}\coth{4}},\qquad \gamma_2^0 \colonequals \tfrac{2\pi (-\tanh{3}+\tanh{1}+\coth{2})}{-(\tanh{3})^2+(\tanh{1})^2+\coth{2}\coth{4}}.
\end{equation}
For a precise statement, see Corollary~\ref{2Pcorolaryfinal}.

This type of point vortex motion is observed, for example in the wake of two stationary cylinders~\cite{2Pflappingfoil, fayed2011visualization}, and they have been mathematically treated in~\cite{BasuStemler2015, StemlerBasu2Pwakes, Stremler20142P, Stemler2P2C, ArefExoticvortexwakes}. The family of $2$P hollow vortices that we construct bifurcates from the point vortex configuration 
 \begin{equation}\label{STintrocoordinates}
        \begin{aligned}
         f(0) &= \id, \quad Q(0) =0, \quad c(0) = 1, \quad \gamma_2(0) = \gamma_2^0 ,\quad \gamma_4(0) =-\gamma_2^0 \\
         w(0)&=\frac{\gamma_1^0}{2\pi i } \log \Bigg( \frac{\sin (\placeholder-i)}{{\sin(\placeholder + i)}} \Bigg) +\frac{\gamma_2^0}{2\pi i } \log \Bigg( \frac{\cos (\placeholder-2i)}{{\cos(\placeholder + 2i)}} \Bigg) .
\end{aligned}
\end{equation}

These three examples are all non-degenerate in the sense that Jacobian of $\mathcal{V}$ has full rank at $\Lambda_{\textup{vk}}^0$, $\Lambda_{\textup{ar}}^0$, and $\Lambda_{\textup{2P}}^0$, so one could simply apply Theorem~\ref{global desingularization theorem} to obtain large families of hollow vortex solutions. However, by enforcing additional symmetries, we are able to fix more of the parameters. For example, this ensures that even the non-perturbative hollow vortices along $\mathscr{C}_{2\textup{P}}$ are indeed $2$P as the vortex centers are unchanging and each pair consists of two counter-rotating vortices.  

\subsection{Plan of the paper}
\label{outline section}
Let us now outline the main steps of the argument. The overall strategy is much the same as in the non-periodic case~\cite{chen2023desingularization}, but there are nontrivial modifications that must be made in the periodic setting. Section~\ref{Layer-potential representations} introduces the layer-potential operators that will be our primary means of fixing the domain. These are Cauchy-type integral operators with a kernel of the form $\cot{(\placeholder)}$. We prove a number of important properties of these mappings that lay the foundation for the analysis in the remainder of the paper. A recurring theme is that in the periodic setting, the behavior at infinity is less constrained. Indeed, unlike the finite vortex configurations considered in~\cite{chen2023desingularization}, here the relative velocity need not vanish at infinity, and it can have distinct (constant) limits as $\imagpart{\zeta} \to +\infty$ and as $\imagpart{\zeta} \to -\infty$. 

In Section~\ref{formulation section}, we reformulate the hollow vortex problem in a setting amenable to bifurcation theory. Recall that the unknowns are the conformal mapping $f$ and the complex velocity potential $w$, which can be uniquely constructed from their values on the boundary of the conformal domain. The kinematic condition~\eqref{kinematicconditionintro} and dynamic condition~\eqref{bernoulliconditionintro} can likewise be posed as nonlocal equations for the boundary traces. But $\partial\confD_\rho$ is itself varying with $\rho$, so we instead work with functions defined on the domain $\mathbb{T}^M$, where $\mathbb{T}$ is the unit circle in the complex plane. Specifically, we represent $f$ and $w$ using the above mentioned layer-potential operators in terms of real-valued H\"older continuous densities $\mu,\nu \in C^{\ell + \alpha}(\mathbb{T})^M$. The steady hollow vortex problem~\eqref{PHVP} can then be stated as the nonlinear operator equation
$$
\mathscr{F}(\mu, \nu ,Q ,\lambda; \rho) =0. 
$$
for a (real-analytic) mapping $\F$ between certain Banach spaces. At $\rho = 0$, this equation collapses analytically to the steady point vortex problem~\eqref{zero-set mathcal V}.

 In Section~\ref{local section}, we therefore make an implicit function theorem argument to construct families of small, near-circular hollow vortices. The curve of solutions bifurcates from the configuration (0,0,0,$\lambda_0$,0), where $\lambda_0$ is the initial value of $\lambda$ in the steady point vortex configuration $\Lambda_0$. Taking advantage of its diagonal structure, we prove that the linearized operator $D_{(\mu,\nu,Q,\lambda )} \mathscr{F} (0,0,0,\lambda_0; 0)$ is an isomorphism if and only if the (finite-dimensional) Jacobian of the point vortex problem  $D_{\lambda}\mathcal{V}(\Lambda_0)$ is an isomorphism. Thus, in the non-degenerate setting, we obtain a (real-analytic) curve $\mathscr{C}_{\text{loc}}$ parameterized by $\rho$, whose elements correspond to periodic hollow vortex configurations with conformal radius $\rho$ as illustrated in Figure \ref{bifurcationdiagram}. 

Section~\ref{global section} is devoted to extending the local solution curve $\cm_\loc$ to a global branch of solutions $\cm$.  A preliminary use of the global implicit function theorem guarantees the existence of a global curve, but allows for a much larger set of alternatives for the limiting behavior along it: either there is (i) blowup in norm of the densities, (normalized) Bernoulli constants, or parameters, or else (ii) there is ``loss of compactness or Fredholmness'', meaning roughly that the zero-set of $\F$ fails to be locally compact or that $\F$ is not semi-Fredholm along it. The main task is then to pare this down to just~\eqref{blowupintro}. Here, the periodicity complicates a number of the arguments.

Finally, in Section~\ref{applications}, we apply these theorems to the von Kármán vortex street, vortex array, and $2$P examples mentioned above. In part, this relies on a new bound on the wave speed for singly periodic hollow vortex configurations.

\subsection*{Notation} \label{notation}
Here we lay out some notational conventions for the remainder of the paper.  Let $\mathbb{T}$ be the unit circle in the complex plane and $\mathbb{F} = \mathbb{R}$ or $\mathbb{C}$.  For $\ell \geq 0$ and $\alpha \in (0,1)$, we denote by $C^{\ell+\alpha}(\mathbb{T},\mathbb{F})$ the space of  H\"older continuous functions of order $\ell$, exponent $\alpha$, and having domain $\mathbb{T}$ that take values in $\mathbb{F}$; if $\mathbb{F}=\mathbb{R}$ we just write $C^{\ell+\alpha}(\mathbb{T})$. Every $\varphi \in C^{\ell+\alpha}(\mathbb{T})$ admits a unique power series representation 
\[
	\varphi(\tau) = \sum_{m \in \mathbb{Z}} \widehat \varphi_m \tau^m, \qquad \textrm{where} \quad \widehat\varphi_m \colonequals  \frac{1}{2\pi} \int_{\mathbb{T}} \varphi(\tau) \overline{\tau^m} \, d\theta = \frac{1}{2\pi i} \int_{\mathbb{T}} \varphi(\tau) \overline{\tau^{m+1}} \, d\tau.
\] 
Here and elsewhere, we write $\tau = e^{i\theta}$ when parameterizing $\mathbb{T}$ by arc length.  Note that because these are real-valued functions, the coefficients must obey $\widehat\varphi_{-m} = \overline{\widehat\varphi_m}$.  For $m \geq 0$, let $\proj_m$ denote the projection
\begin{equation}
  \label{definition projection operator}
  \proj_m \colon C^{\ell+\alpha}(\mathbb{T}) \to C^{\ell+\alpha}(\mathbb{T}), \quad 
  (P_m\varphi)(\tau) \colonequals  
    \left\{ \begin{aligned}
      \widehat\varphi_m \tau^m + \widehat\varphi_{-m} \tau^{-m} & \qquad  \textrm{if }m \neq 0 \\
       \widehat\varphi_0 & \qquad  \textrm{if } m = 0.
    \end{aligned}
    \right.
\end{equation}
and set
\[ 
	\proj_{\leq m} \colonequals  \proj_0 + \cdots + \proj_m, \qquad \proj_{> m} \colonequals  1-\proj_{\leq m}.
\]
We will often work with the spaces
\[
	\mathring{C}^{\ell+\alpha}(\mathbb{T},{\mathbb{F}}) \colonequals  \proj_{> 0} C^{\ell+\alpha}(\mathbb{T},{\mathbb{F}}),\qquad 
\]
which corresponds to elements of $C^{\ell+\alpha}(\mathbb{T})$ normalized to have mean $0$.

Finally, throughout the paper, we will use the Wirtinger derivative operators
\[
	\partial_z \colonequals  \frac{1}{2} \left( \partial_x - i \partial_y \right), \qquad
	\partial_{\overline z} \colonequals  \frac{1}{2} \left( \partial_x + i \partial_y \right).
\]
In particular, a function $f = f(z,\bar z)$ is holomorphic precisely when $\partial_{\bar z} f = 0$ and antiholomorphic provided $\partial_z f = 0$.  When there is no risk of confusion, these may also be denoted using subscripts, thus $f_z = \partial_z f$ and so on.  Primes will be reserved for denoting Wirtinger derivatives of functions with domain $\mathbb{T}$.

\section{Layer-potential representations}
\label{Layer-potential representations}

Recall that our strategy is to reformulate the hollow vortex problem in terms of the traces of $f$ and $w$ on the boundary of the conformal domain using periodic layer potential operators. Of course, for the conformal description to be well-defined, it is necessary that each of the components of $\partial\confD_{\rho}$ is disjoint. It will be convenient to have a family of open sets in parameter space that enforce this requirement quantitatively. With that in mind, for each $\delta > 0$, let 
\begin{equation}\label{defofmathcalU}
\mathcal{U}_{\delta}\colonequals \{ (\rho, \Lambda) \in \mathbb{R}\times \mathbb{P} : \min_{j \neq k } |\zeta_j - \zeta_k| > 2|\rho | + \delta,~  \min_{j,k} \{ |\zeta_k-\zeta_j- T| \}>2|\rho| +\delta  \},
\end{equation}
and set $\mathcal{U} \colonequals \bigcup_{\delta > 0} \mathcal{U}_{\delta }$.

Let us now introduce the layer potential operators, which are an adaptation of those in~\cite{chen2023desingularization} to the singly-periodic setting.  For $(\rho, \Lambda) \in \mathcal{U}$ and real-valued densities $ \mu = (\mu_1, \ldots, \mu_M) \in \mathring{C}^{\ell+\alpha}(\mathbb{T})^M$, we define the layer-potential operator
\begin{equation}
\label{definition for Zrho}
 \mathcal{Z}^{\rho}(\Lambda)[\mu] (\zeta)\colonequals \sum_{k=1}^{M}\frac{1}{2 T i}\int_{\mathbb{T}}\mu_k(\sigma)\cot\Big(\frac{\pi}{T}(\rho \sigma +\zeta_k^0-\zeta) \Big)\rho \,  d\sigma.
\end{equation}
It is clear that the right-hand side of \eqref{definition for Zrho} constitutes a single-valued holomorphic in $\mathcal{D}_\rho$ that is periodic with period $T$.  Moreover, for fixed $(\rho,\fullparam) \in \mathcal{U}$ with $|\rho| > 0$, we have by Privalov's theorem that 
$$
\mathcal{Z}^\rho(\fullparam) \quad \textrm{is bounded} \quad C^{\ell+\alpha}(\mathbb{T})^M \to C^{\ell+\alpha}(\overline{\confD_\rho},\mathbb C),
$$
and this bound is uniform on compact subsets of $\mathcal{U}$. To keep the notation manageable, we will often suppress the dependence on $\Lambda$ when there is no risk of confusion, and write $\mathcal{Z}^\rho \mu$ rather than $\mathcal{Z}^\rho[\mu]$.

\begin{remark}\label{symmetryofZwrtmurho}
A simple change of variables reveals that: 
$\mathcal{Z}^{\rho} \mu = \mathcal{Z}^{-\rho}[\mu(-\placeholder)]$. 
\end{remark}

For each $k = 1, \ldots, M$, we likewise define the trace operator
\begin{equation}
\begin{aligned}
\label{definition for Zkrho}
    \mathcal{Z}_k^{\rho}[\mu](\tau) :=\mathcal{Z}^{\rho}(\zeta_k+\rho \tau) &=\frac{1}{2iT}\int_{\mathbb{T}}(\mu_k(\sigma)-\mu_k(\tau))\cot{\left (\frac{\pi }{T}\rho(\sigma-\tau) \right)} \rho \, d\sigma\\ &\qquad+
    \sum_{j\neq k}\frac{1}{2iT}\int_{\mathbb{T}}\mu_j(\sigma)\cot{\left(\frac{\pi }{T}(\rho(\sigma-\tau)+\zeta_j-\zeta_k) \right)} \rho \,  d\sigma \\
    &\equalscolon \mathcal{C}^{\rho}\mu_k + \sum_{j\neq k}\frac{1}{2iT}\int_{\mathbb{T}}\mu_j(\sigma)\cot{\left(\frac{\pi }{T}(\rho(\sigma-\tau)+\zeta_j-\zeta_k) \right)} \rho \, d\sigma.  
\end{aligned}
\end{equation}
where the first line follows from the Sokhotski--Plemelj formula. By standard potential theory, for fixed $\rho$ and $\Lambda$, 
\begin{equation}\label{boundednessofZkrho}
    \mathcal{Z}^\rho_k(\Lambda) \quad \textrm{is bounded} \quad C^{\ell+\alpha}(\mathbb{T})^M \to C^{\ell+\alpha}(\mathbb{T},\mathbb C).
\end{equation}
Also assuming that $\delta<|\realpart(\zeta_j-\zeta_k)| < T$ for all $j,k=1,\ldots M$ we have
\begin{equation}\label{local expression for Zkrhoquick}
\mathcal{Z}_{k}^{\rho} [\mu] (\tau) = \mathcal{C}^{\rho}[\mu_k](\tau) + O(\rho) \quad \text{ in } C^{\ell+\alpha}(\mathbb{T}).
\end{equation} 
We defer the proof of this fact to Proposition~\ref{firstZ_krhoassymprotic}.

Consider next the behavior of $\mathcal{Z}^{\rho}$ as $\rho \to 0$. For $0 < |\rho| \ll 1$ and $j \neq k$, it is easy to see that
\begin{align*}
	& \left |\frac{1}{2iT}\int_{\mathbb{T}}\mu_j(\sigma)\cot{\left(\frac{\pi }{T}(\rho(\sigma-\tau)+\zeta_j-\zeta_k) \right)} \rho \,  d\sigma \right| \\
	& \qquad \lesssim  \| {\mu}\|_{L^{\infty}(\mathbb{T})} \sup_{\kappa \in \mathbb{T}-\mathbb{T}}\left|{\frac{1}{T}}\cot{\left({\frac{\pi}{T}}(\zeta_k-\zeta_j+{\rho}{\kappa})\right)}\right||\rho| \longrightarrow 0
\end{align*}
as $\rho \to 0$.  On the other hand, using the Laurent series for the cotangent
\begin{equation}\label{cotangentseries}
\cot{z} = \frac{1}{z} + \sum_{n\geq 1} a_nz^{n}, \qquad \textrm{for } |z| < \pi,
\end{equation}
it is straightforward to see that as $\rho\to 0$
$$
\mathcal{C}^{\rho}  \longrightarrow \mathcal{C}^0 \quad \text{in }  \mathcal{L} (C^{\ell+\alpha}(\mathbb{T})^M, C^{\ell+\alpha}(\mathbb{T},\mathbb{C}) ) \quad \text{ where }\quad \mathcal{C}^0 \varphi \colonequals \frac{1}{2Ti} \int_{\mathbb{T}} \frac{\varphi(\sigma)-\varphi(\placeholder)}{\sigma-\placeholder}  \, d\sigma,
$$
for all $\varphi \in C^{\ell+\alpha}(\mathbb{T}; \mathbb{C})$.  Unsurprisingly, $\mathcal{C}^0$ is precisely the Cauchy-type integral operator that arose as the singular part of the layer potential in the non-periodic case~\cite{chen2023desingularization}. 
Combining that with \eqref{local expression for Zkrhoquick} we obtain the rather useful fact 
\begin{equation}\label{local expression for Zkrho}
\mathcal{Z}_k^{\rho}[\mu] =\mathcal{C}^{0}[\mu_k] + O(\rho) \quad \text{ in } C^{\ell+\alpha}(\mathbb{T}).
\end{equation}

\begin{lemma}\label{rangeClemma}
For each $m \in \mathbb{Z}$, the following formula holds:
$$
\mathcal{C}^{\rho}\tau^{m} = \left\{ 
	\begin{aligned}
	0 & \qquad m \geq 0 \\
	\rho^{|m|}\left(\frac{\pi}{T}\right)^{|m|} \frac{(-1)^{|m|}}{(|m|-1)!}\cot^{(|m|-1)}{\left(\frac{\pi \rho}{T}\tau \right)} & \qquad m < 0.
	\end{aligned} 
	\right.
$$
\end{lemma}
\begin{proof}
This follows from repeated integration by parts, residue calculus, and the Laurent series expression for the cotangent~\eqref{cotangentseries}.
\end{proof}
\begin{remark}\label{rangeCremark}
For $m > 0$ one can alternative express
\[
	\mathcal{C}^{\rho} [\tau^{-m}] = -\tau^{-m} +(-1)^m \sum_{n\geq m-1} \Big(\frac{\pi\rho }{T} \Big)^{n+1}{n \choose m-1}a_n \tau^{n-m+1}
\]
with the $a_n$ being the coefficients of the Laurent series~\eqref{cotangentseries}. In particular, unlike $\mathcal{C}^0$, when $\rho \neq 0$, $\mathcal{C}^\rho$ is not simply a projection onto the negative modes. However, from Lemma~\ref{rangeClemma} we can deduce that $\mathcal{C}^0 \mathcal{C}^\rho = -\mathcal{C}^0$. Moreover,  
\[
	\range{\mathcal{C}^{\rho} \partial_{\tau} \big|_{C^{\ell+\alpha} (\mathbb{T,\mathbb{C}})} } \subset \lspan_{\mathbb{C}}\{\tau^m\text{ : } m\neq -1 \}.
\]
These two facts will be extremely helpful in the calculations to follow in Section~\ref{local section}. 
\end{remark}

\begin{proposition}\label{proof of communication with tau derivative}
Let $\mu \in C^{\ell+\alpha}(\mathbb{T})^M$ and $(\rho,\Lambda )\in \mathcal{U}$ be given. 
\begin{enumerate}[label=\rm(\alph*)]
\item \label{proof of communication with tau derivativeA} $\mathcal{Z}^{\rho}[\mu]$ is analytic in $\confD_\rho$ and $C^{\ell+\alpha}(\overline{\confD_\rho})$ while $\mathcal{Z}_k^\rho \mu \in C^{\ell+\alpha}(\mathbb{T})$, and we have the following commutation identities
\be
\label{commutation with tau derivatives}
\partial_{\zeta}\mathcal{Z}^{\rho}[\mu] =\frac{1}{\rho}\mathcal{Z}^{\rho}[\mu^\prime],\qquad \partial_{\tau} \mathcal{Z}_{k}^{\rho}[\mu]=
\mathcal{Z}_{k}^{\rho}[\mu'].
\ee
\item \label{Zrho is bounded} Moreover, for $\ell \geq 2$, 
$$
\partial_{\zeta}\mathcal{Z}^{\rho}[\mu],~\partial_{\zeta}^2\mathcal{Z}^{\rho}[\mu]\to 0, \qquad \text{as } \imagpart\zeta \to \pm \infty.
$$
\end{enumerate}
\end{proposition}
\begin{proof}
For the second identity in part~\ref{proof of communication with tau derivativeA}, an immediate application of Lemma~\ref{rangeClemma} yields: 
$$
\partial_{\tau} \mathcal{C}^{\rho}[\mu_k](\tau)=
\mathcal{C}^{\rho}[\mu_k'](\tau)
.$$
So we may only prove it for the second term in the definition of $\mathcal{Z}_{k}^{\rho}$ in \eqref{definition for Zkrho}. But this is accomplished by a simple integration by parts argument. For the first identity, we write $\partial_{\zeta}$ as $-\frac{1}{\rho}\partial_{\tau}$ and integrate by parts. To prove part~\ref{Zrho is bounded}, we exploit the decay of $\csc{\zeta}$ as $\imagpart\zeta\to \pm \infty$. 
\end{proof}

It will be useful later to recover $\mu_k$ given $\mathcal{Z}_k^{\rho}$. This can be done in the spirit of \cite{chen2023desingularization}, however we will also need to exploit the the fact that $\mathcal{C}^{0}\mathcal{C}^{\rho} = - \mathcal{C}^{0} $ which follows directly from Lemma \ref{rangeClemma}. With this in mind, and \eqref{commutation with tau derivatives} it quickly follows that, 
$$
\mu_k =2 \realpart (\mathcal{C}^{0} \mathcal{Z}_{k}^{\rho}[\mu]  ).
$$
Thus, since $\mathcal{C}^{0}$ is bounded in $C^{\alpha}(\mathbb{T})$, and $L^{p}(\mathbb{T})$ for all $1< p < \infty$, we have that, 
\begin{equation}\label{inversionforZk}
\| \partial_{\tau}^{\ell}\mu_k\|_{C^{\alpha}(\mathbb{T})} \leq C_{\alpha}\|{\mathcal{Z}_k^{\rho}\partial_{\tau}^{\ell}\mu_k}\|_{C^{\alpha}(\mathbb{T})}, \quad \|{\partial_{\tau}^{\ell}\mu_k}\|_{L^p(\mathbb{T})} \leq C_{p}\|{\mathcal{Z}_k^{\rho}\partial_{\tau}^{\ell}\mu_k}\|_{L^{p}(\mathbb{T})}.
\end{equation}

\begin{proposition}\label{firstZ_krhoassymprotic}
 There exists a constant $C > 0$ depending on positive lower bounds for $\delta$ and $|\rho|$ so that for each $(\Lambda, \rho) \in \mathcal{U}_\delta$ and $\mu \in C^{\ell+\alpha} (\mathbb{T},\mathbb{F})^M$, we have
 \begin{enumerate}[label=\rm(\alph*)]
\item\label{firstZ_krhoassymprotic1}
$
\|\mathcal{Z}_k^{\rho}\mu - \mathcal{C}^{\rho} \mu_k\|_{C^{\ell+\alpha}(\mathbb{T},\mathbb{C})} \leq C\rho \|{\mu}\|_{C^{\ell}(\mathbb{T},\mathbb{F})^M} 
$
\item\label{firstZ_krhoassymprotic2}
$
\| \mathcal{C}^{\rho} \mu_k - \mathcal{C}^{0}\mu_k \|_{C^{\ell+\alpha}(\mathbb{T},\mathbb{C})} \leq C \rho \|\mu_k \|_{C^{\ell}(\mathbb{T},\mathbb{F})}
$
\end{enumerate} 
for $\mathbb{F}= \mathbb{R}$ or $\mathbb{C}$.
\end{proposition}
\begin{proof} Throughout the proof we use $C$ to denote a generic positive constant depending on the parameters as in the statement of the proposition. For part~\ref{firstZ_krhoassymprotic1}, recall that
\begin{align*}
\mathcal{Z}_k^{\rho}\mu - \mathcal{C}^{\rho} \mu_k &= \frac{1}{2Ti} \sum_{j\neq k } \int_{\mathbb{T}} \mu_{j}(\sigma) \cot{\left(\frac{\pi}{T}\rho(\sigma-\tau) + \zeta_j-\zeta_k \right)} \rho \, d\sigma  \equalscolon \frac{1}{2Ti} \sum_{j\neq k } \int_{\mathbb{T}} \mu_{j}(\sigma)h_j(\sigma-\tau)\rho \, d\sigma.
\end{align*}
Observe that $h_j \in C^{\infty}(\mathbb{T}; \mathbb{C})$ with $\| h_j \|_{C^\ell} < C_\ell$ for each each $\ell \geq 0$. Thus, 
$$
	\| \mathcal{Z}_k^{\rho} \mu - \mathcal{C}^{\rho} \mu_k \|_{C^\alpha(\mathbb{T}; \mathbb{F})} \lesssim  C \rho \|\mu \|_{C^0}.
$$
Then, the bound in~\ref{firstZ_krhoassymprotic1} follows from the commutation identity~\eqref{commutation with tau derivatives}. The argument for part~\ref{firstZ_krhoassymprotic2} is similar, so we omit it.
\end{proof}

\subsection{Injectivity and conformality}\label{infectivity conformality section} 
In this section, we provide a few important technical lemmas that allows us to infer conformality and injectivity of the conformal mapping $f$ from a winding condition and its boundary behavior. We also give conditions guaranteeing that $\partial f(\mathcal{D}_{\rho})$ consists of Jordan curves as in Figure \ref{physicaldomain}. 

Throughout this section we let 
\[
	g \colonequals f - \id.
\]
We will assume that  $g\in C^{2}(\overline{\mathcal{D}}, \mathbb{C})$, $g$ is analytic on $\confD_\rho$,  $T$-periodic, and satisfies
\begin{equation}\label{limitofg}
\partial_{\zeta }g(\zeta),~\partial_{\zeta}^2g(\zeta) \to 0 , \qquad \textrm{as } \imagpart\zeta \to \pm \infty.
\end{equation}
Note that due to Proposition~\ref{proof of communication with tau derivative}, the decay above is satisfied by $\mathcal{Z}^\rho \mu$ for any $\mu \in \mathring{C}^{\ell+\alpha}(\mathbb{T})^M$ with $\ell \geq 2$. 
In what follows we use the notation
\begin{equation}
    N(f = a; \, \mathcal{O} ) \colonequals \# \left(  f^{-1}(a) \cap \mathcal{O}  \right) \in \mathbb{N} \cup \{ +\infty\},
\end{equation}
for $\mathcal{O} \subset \mathbb{C}$ and where $\#$ denotes counting measure.
\begin{lemma}[Conformality] \label{complexlemmaconformality}
Let $f$ and $g$ be given as above and $(\Lambda , \rho) \in \mathcal{U}$. Assume that $\partial_{\zeta}f\neq 0$ on $\partial \mathcal{D}_{\rho}$ and $f$ satisfies the winding condition
\begin{equation}\label{integralcondition1}
    \frac{1}{2\pi i }\sum_{k=1}^{M} \int_{\partial B_{\rho} (\zeta_k)} \frac{\partial_{\zeta}^2f }{\partial_\zeta f}\, d\zeta  =0 ,
\end{equation}
then $\partial_{\zeta } f$ is non-vanishing throughout $\overline{\mathcal{D}}_{\rho}$. 
\end{lemma}
\begin{proof}
    Since $\partial_{\zeta} f = 1 + \partial_{\zeta} g $ is $T$-periodic, it suffices to show that it is non-vanishing on any strip of width $T$. Without loss of generality we may assume that $\bigcup_{k=1}^{M} B_{\rho}(\zeta_k)$ lies in 
    $$
    \Omega_n  = \{ \zeta \in \mathbb{C} \text{ : } 0 < \realpart \zeta < T,\, -n < \imagpart \zeta < n \}.
    $$
for large enough $n$; we verify the validity of this assumption in Remark \ref{Inegrationstrip}. 
The roots of analytic functions are isolated so we assume that $\partial_{\zeta} f$ has no roots on $\partial \Omega_n$ for all $n$. A standard contour integral argument and the argument principle gives 
\begin{align*}
   N(\partial_{\zeta}f =0 \text{; } \Omega_n) &= -\frac{1}{2\pi i } \sum_{k=1}^{M} \int_{\partial B_{\rho}(\zeta_k)} \frac{\partial_{\zeta}^2 f}{ \partial_{\zeta}f} \, d\zeta + \frac{1}{2\pi i } \int_{\partial \Omega_n} \frac{\partial_{\zeta}^2f }{\partial_{\zeta }f }\, d\zeta  =\frac{1}{2\pi i } \int_{\partial \Omega_n} \frac{\partial_{\zeta}^2f }{\partial_{\zeta }f }\, d\zeta,
\end{align*}
where the second equality follows by hypothesis~\eqref{integralcondition1}. The function $\partial_{\zeta}f$ is $T$-periodic so the integrals over the left and right line segments of $\partial \Omega_n$ cancel, thus  
$$
 N(\partial_{\zeta } f=0\text{; }\Omega_n) = \frac{1}{2 \pi i } \int_{\mathpzc{B}_{\,n}}\frac{\partial_{\zeta}^2 f(\zeta)}{\partial_{\zeta } f(\zeta)}\, d\zeta + \frac{1}{2\pi i }\int_{\mathpzc{T}_{\,n}}\frac{\partial_{\zeta}^2 f(\zeta)}{\partial_{\zeta } f(\zeta)}\, d\zeta 
$$
where $\mathpzc{T}_{\,n},\mathpzc{B}_{\,n} $ are the top and bottom line segments of the rectangle $\Omega_n$, from which they also inherit their orientation. It quickly follows that
\begin{align}\label{quickusecomplex}
 N(\partial_{\zeta } f=0\text{; }\Omega_n) \leq \frac{T}{2\pi} \sup_{\mathpzc{T}_{\,n}\cup \mathpzc{B}_{\,n}} \Bigg|\frac{\partial_{\zeta}^2f}{\partial_{\zeta} f}\Bigg|.
\end{align}
Since $\partial_{\zeta} f = 1+ \partial_{\zeta}g \text{ and }\partial_{\zeta}^2f =\partial_{\zeta}^2g $, the limiting conditions \eqref{limitofg} imply that that for $n \gg 1$ we have that 
\begin{equation*}
    N(f = 0 \text{; } \Omega_n) \ll 1.
\end{equation*}
But $N(f = 0 \text{; } \Omega_n )$ takes non-negative integer values thus we conclude that $N(f = 0 \text{; } \Omega_n ) =0 $ for big enough $n$, which finishes the proof.
\end{proof}

\begin{remark}\label{Inegrationstrip}
By the definition of $\mathcal{U}$~\eqref{defofmathcalU} and periodicity, we can always deform the sides of the strip $\Omega_n$ so that it has a constant width of $T$ and contains $\bigcup_{k=1}^{n}B_{\rho}(\zeta_k)$ and so that $\partial\Omega_n \subset \confD_\rho$. At most, this increases the arclength of the lateral sides by $O(\rho)$.
\end{remark}

\begin{lemma}\label{complexlemmainjectivity}
In the setting of Lemma \ref{complexlemmaconformality},  further assume that $f|_{\partial{\mathcal{D}_{\rho}}}$ is injective and that none of the Jordan curves $\Gamma_k^n = f(B_{\rho}(\zeta_k) + nT)$ encloses another for all $n\in \mathbb{Z}$ and $k=1,\ldots ,M$. Then, $f$ is globally injective
\end{lemma}
\begin{proof}
We prove the lemma for $T=1$ and infer it for $T>0$ such that $(\rho,\Lambda)\in \mathcal{U}$. Consider the set, 
\begin{equation}\label{integralconditionprooflemma}
    \mathscr{D} = \Big\{z\in \mathbb{C} : \int_{\Gamma_k^n} \frac{d\eta}{\eta-z} =0 \text{ }\text{ for all } n\in \mathbb{Z} ,\, k=1,\ldots , M\Big\}
\end{equation}
By the Jordan curve theorem, and the assumption that none $\Gamma_k^n$ enclose one another it follows that $\mathscr{D}$ is a connected and unbounded set such that: 
$$
\partial\mathscr{D} = \bigcup_{k=1}^{M}\bigcup_{n\in \mathbb{Z}} \Gamma_k^n.
$$
Our goal is to show that $f \colon \overline{\mathcal{D}_{\rho}}\to \overline{\mathscr{D}}$ is bijective. Let $z\in \mathscr{D}$ be given,  which we fix for the remainder of the proof. It suffices to prove 
$$
N(f = z ;\, \mathcal{D}_{\rho}) = 1.
$$
Consider the rectangle on the complex plane
\begin{equation*}
    \Omega_n = \{ z\in \mathbb{C} \text{ : } -n<\realpart z < n+1,~ |\imagpart{z}| < \sqrt{n}\} 
\end{equation*}
and the punctured rectangle $\widetilde{\Omega}_n = \mathcal{D}_{\rho} \cap \Omega_n.$ As the function $f$ is a bounded perturbation of the identity,
$$
N (f=z;\, \mathcal{D}_{\rho}) = N(f = z;\, \widetilde{\Omega}_n) \quad \text{ for some } n\gg1 .
$$
We split the rectangle ${\Omega}_n$ into adjacent subrectangles $\{\Omega_{n}^{m}\}_{m=-n}^{n}$, such that: 
$$
\Omega_{n}^{m} = \{z\in \mathbb{C} \text{ : } m<\realpart z<m+1,~ |\imagpart{z}| < \sqrt{n} \},
$$
and, do likewise for the punctured rectangle: $\widetilde{\Omega}_{n}^{m}=\Omega_{n}^{m}\cap \mathcal{D}_{\rho}.$ Applying the argument principle gives
\begin{align}\nonumber
N(f=z;\, \widetilde{\Omega}_{n}) &=\frac{1}{2\pi i}\int_{\partial \Omega_n}\frac{\partial_{\zeta}f(\zeta)}{ f(\zeta) - z} \, d\zeta - \frac{1}{2\pi i}\sum_{m=-n}^{n}\sum_{k=1}^{M}\int_{\partial B_{\rho}(\zeta_k)+m}\frac{\partial_{\zeta}f(\zeta)}{ f(\zeta) - z} \, d\zeta \\
	& =\frac{1}{2\pi i}\label{reduction1} \int_{\partial \Omega_n}\frac{\partial_{\zeta}f(\zeta)}{ f(\zeta) - z} \, d\zeta,
\end{align}
where the second equality follows from the definition of $\fluidD$ in~\eqref{integralconditionprooflemma}. 

We will show that $|N(f=z;\, \widetilde{\Omega}_{n})-1| \ll 1$ for large enough $n$, from which the statement follows. The boundary of $\Omega_n$ consists of four line segments $\mathfrak{T}_n,\mathfrak{R}_n,\mathfrak{B}_n$ and $\mathfrak{L}_n$ for the top, right, bottom and left portion respectively, which inherit their orientation from $\partial\Omega_n$ which is counter clockwise. By~\eqref{reduction1}, it will therefore suffices to prove that
\be
\label{the four integrals}
	\left| {\frac{1}{2\pi i}}\int_{\mathfrak{B}_n} \frac{\partial_{\zeta}f(\zeta)}{ f(\zeta) - z} \, d\zeta  -\frac{1}{2}\right|, ~
    \left|{\frac{1}{2\pi i}} \int_{\mathfrak{T}_n} \frac{\partial_{\zeta}f(\zeta)}{ f(\zeta) - z} \, d\zeta  -\frac{1}{2}\right|,~ 
    \left| {\frac{1}{2\pi i}}\int_{\mathfrak{L}_n \cup \mathfrak{R}_n} \frac{\partial_{\zeta}f(\zeta)}{ f(\zeta) - z} \, d\zeta \right|  \ll 1
\ee
for $n \gg 1$.
 We first illustrate how to obtain the bound for the integral over $\mathfrak{B}_n$; the argument for the integral over $\mathfrak{T}_n$ is identical.   Using a simple change and after some elementary computations we arrive at 
\begin{equation}\label{onestepcloser}
	\int_{\mathfrak{B}_n} \frac{\partial_{\zeta}f(\zeta)}{ f(\zeta) - z} \, d\zeta = \int_{\mathfrak{B}_n^0}\partial_{\zeta}f(\zeta) \left(\frac{1}{h(\zeta)} + \sum_{k=1}^{n} \frac{2 h(\zeta)}{h(\zeta)^2-k^2} \right) \, d\zeta,
\end{equation}
where $h(\zeta) \colonequals \zeta + g(\zeta) - z$ and $\mathfrak{B}_n^0$ is the portion of $\mathfrak{B}_n$ with real part between $0$ and $1$.  Our goal is to use the point-wise cotangent formula \eqref{cotexp} in a uniform manner. We can, in principle, do this on sets where $h(\zeta)$ is bounded uniformly away from the positive integers. With that in mind, note for $\zeta \in \mathfrak{B}_n^0$, we for $n \gg 1$, 
\[
	|h(\zeta)| \leq |\zeta| + |g(\zeta) | +|z| \leq 2\sqrt{n},
\]
and
\begin{align*}
    |h^2(\zeta)-k^2|& \geq |\realpart (h^2(\zeta) -k^2)|  \\
    & = |\realpart (\zeta^2-k^2)|-2|\realpart \zeta (g(z)-z) | -|\realpart (g(z)-z)^2| & \\
    & \geq \inf_{t\in [0,1]}|\realpart\big((t + i \sqrt{n} )^2 - k^2 \big)| - 8\sqrt{n} \geq \frac{1}{2}k^2.
\end{align*}
Thus, along $\mathfrak{B}_n^0$,
\begin{align*}
    \sum_{k=n+1}^{\infty} \frac{|2h(\zeta)|}{|h(\zeta)^2 - k^2|} \leq  4\sum_{k=n+1}^{\infty}\frac{ \sqrt{n} }{k^2}  
    & = O\left(\frac{1}{\sqrt{n}} \right).
\end{align*}
So, for $n \gg 1$, we have 
\begin{equation*}
\Big| \pi \cot(\pi h(\zeta)) - \frac{1}{h(\zeta)} - \sum_{k=1}^{n}\frac{2h(\zeta)}{h(\zeta)^2-k^2} \Big| = O\left(\frac{1}{\sqrt{n}} \right) \qquad \textrm{for all } \zeta \in \mathfrak{B}_n^0.
\end{equation*}
Now, because $h(\zeta) - \zeta = O(1)$ and $\partial_\zeta f(\zeta) -1 = \partial_\zeta g(\zeta) \to 0$ as $\imagpart{\zeta} \to -\infty$, it follows that $\cot{(\pi \placeholder)} \circ h \to i$ in $L^\infty(\mathfrak{B}_n^0)$ as $n \to \infty$. From \eqref{onestepcloser}, we therefore infer that  
\begin{align*}
    \left| {\frac{1}{2\pi i}} \int_{\mathfrak{B}_n} \frac{\partial_{\zeta}f(\zeta)}{ f(\zeta) - z} \, d\zeta - \frac{1}{2} \right| & \longrightarrow 0 \qquad \textrm{as } n \to \infty,
\end{align*}
as desired. 

Consider lastly the bound for the integral along $\mathfrak{R}_n$ in~\eqref{the four integrals}. It is straightforward to see that  
\begin{align*}
     \left|  \int_{\mathfrak{R}_n} \frac{\partial_{\zeta}f(\zeta)}{f(\zeta) -z}\, {d\zeta} \right|  & \leq 2 \sqrt{n} \sup_{\zeta\in \mathfrak{R}_n}\frac{|\partial_{\zeta}f(\zeta)|}{|f(\zeta)-z|} \leq 2 \sqrt{n} \left( \frac{1+\| \partial_{\zeta}g\|_{L^\infty(\confD_\rho)} }{\inf_{\mathfrak{R}_n} |f-z|} \right) = O\left( \frac{1}{\sqrt{n}} \right)
\end{align*}
as $n \to \infty$. The integral in~\eqref{the four integrals} along $\mathfrak{L}_n$ is handled in the same way.
\end{proof}

\section{Nonlocal formulation} \label{formulation section}
In this section, we use the layer-potential operators introduced in Section~\ref{Layer-potential representations} to reformulate the steady hollow vortex problem~\eqref{PHVP} as a nonlocal system for densities representing the conformal mapping $f$ and complex potential $ w$. Our treatment follows that same overarching strategy as Chen, Walsh, and Wheeler~\cite{chen2023desingularization}, but adapting it to the periodic setting introduces sone new technical difficulties due to the more complicated kernel for the layer potentials.
 
Recall that we wish to find a conformal and injective mapping $f$ such that $f(\mathcal{D}_{\rho})$ is the physical domain, and a conformal velocity potential $w$ defined on $\mathcal{D}_{\rho}$ such that $\mathsf{w} = w \circ f^{-1}$ is the velocity potential on the physical domain. The main observation is that the streamlines for a collection of small hollow vortices should be qualitatively similar to those of the starting point vortex configuration. Thus, we take $f$ to be a near-identity conformal mapping and suppose that $w$ is a (single-valued) holomorphic perturbation of the potential for the point vortices. Specifically, for $f$ we set
\begin{equation}\label{ansatzforf}
f = \id + \rho^2 \mathcal{Z}^{\rho}\mu,
\end{equation}
where $\mathcal{Z}^{\rho}$ was defined in \eqref{definition for Zrho} and the real-valued density $\mu = (\mu_1,\ldots , \mu_M) \in \mathring{C}^{\ell+{\alpha}} (\mathbb{T})^M$ becomes a new unknown. Likewise, on the complex potential we impose the ansatz   
\begin{equation}\label{ansatzforw}
w = w^0 + \rho\mathcal{Z}^{\rho}\nu
\end{equation}
for an unknown real-valued density $\nu = (\nu_1,\ldots , \nu_M) \in \mathring{C}^{\ell+{\alpha}} (\mathbb{T})^M $. Recall that $w^0$ is the complex potential for the point vortex configuration and given by 
\[
	w^0 \colonequals \sum_{k=1}^{M}\frac{\gamma_k}{2\pi i }\log \sin \left(\frac{\pi }{T}(\placeholder-\zeta_k) \right).
\]
We will usually suppress its dependence on $\lambda$ for the sake of readability. The factors of $\rho^2$ and $\rho$ appearing in~\eqref{ansatzforf} and \eqref{ansatzforw} anticipate the scaling of the problem, and are used to ensure that the abstract operator equation governing $\mu$ and $\nu$ converge smoothly to the point vortex system~\eqref{zero-set mathcal V} as $\rho \to 0$.

\begin{remark}\label{symmetryremark}
According to Remark \ref{symmetryofZwrtmurho} the configurations $(\mu,\nu,\rho)$ and $\mu(-\placeholder),-\nu(-\placeholder),-\rho)$ yield the same conformal mapping $f$ and velocity potential $w$ independently of the periodic steady point vortex configuration $\Lambda $. The difference in signs for $\mu(-\placeholder),\nu(-\placeholder)$ accounts to the normalizing factor $\rho^2$ for $f$ and $\rho$ for $w$. 
\end{remark}

\subsection{Kinematic condition}
The kinematic condition \eqref{kinematicconditionintro} dictates that the relative velocity field is purely tangential along connected components of $\partial \mathscr{D}$. In the conformal variables, this amounts to requiring
\begin{equation}\label{Kinematiconditionlocal}
    \realpart \Bigg( \tau f_{\zeta}\Bigg( \frac{w_{\zeta}}{f_{\zeta}}-c \Bigg)\Bigg)\Bigg|_{\rho\tau +\zeta_{k}}=0,
\end{equation}
for all $\tau\in\mathbb{T}$ and $k=1,\ldots ,M$. To see this, $\tau f_{\zeta}$ is the normal vector along the vortex boundary, whereas
\[
	\overline{\partial_z ( w - cf )} =\frac{\overline{\partial_{\zeta}w}}{\overline{\partial_{\zeta}f}} - \overline{c}
\]
is the complex relative velocity field in the conformal domain.

Now, assuming for the moment that $0<|\rho|\ll 1$, we have from~\eqref{ansatzforf} and \eqref{ansatzforw} that
\begin{align*}
	\partial_\zeta f(\zeta_k + \rho \tau) = 1 + \rho \mathcal{Z}_k^{\rho}[\mu^\prime](\tau),
\end{align*}
and
\begin{align*}
\partial_\zeta w(\zeta_k + \rho \tau) & = \partial_\zeta w^0(\zeta_k + \rho \tau) + \mathcal{Z}_k^\rho[\nu^\prime](\tau) \\
	& = \frac{\gamma_k}{2T i}\cot {\left(\frac{\pi}{T}\rho\tau \right)} +
    \sum_{j\neq k}\frac{\gamma_j}{2Ti}\cot\Big(\frac{\pi}{T}(\zeta_k-\zeta_j+\rho\tau) \Big) + \mathcal{Z}_k^\rho[\nu^\prime](\tau) \\
    	& \equalscolon \frac{\gamma_k}{2\pi\rho \tau i } + \mathcal{V}_{k}^{\rho}(\tau)+c +  \mathcal{Z}_k^\rho[\nu^\prime](\tau)
\end{align*}
for all $\tau \in \mathbb{T}$. Expanding the second term on the right-hand side above more fully reveals that
\begin{equation}\label{Phikrhofirstterm}
\mathcal{V}_k^{\rho}(\tau)=\mathcal{V}_k-\left(\frac{\gamma_k \pi}{6iT^2} + \sum_{j\neq k}\frac{\gamma_j \pi}{2iT^2}  \csc^2\left(\frac{\pi}{T}(\zeta_k-\zeta_j) \right) \right)\rho\tau+O(\rho^2)\quad \text{in } C^{\ell-1+\alpha}(\mathbb{T}).
\end{equation}
Recall that here $\mathcal{V} = \mathcal{V}(\Lambda)$ is the $\mathbb{C}^M$-valued function whose zero-set enforces the Helmholtz--Kirchhoff model~\eqref{definition V_k}. We are continuing our practice of suppressing the dependence of $\mathcal{V}_k^\rho$ on $\Lambda$. So now substituting back into \eqref{Kinematiconditionlocal}, we get the final expression for the kinematic condition as a nonlocal equation for $\mu, \nu, \Lambda$ and $\rho$:
\begin{equation}\label{KinematicConditionlocalfinal}
\realpart \Big(\tau\Big(\mathcal{V}_k^{\rho}+ \mathcal{Z}_k^{\rho}[\nu']-c\rho \mathcal{Z}_k^{\rho}[\mu'] \Big)\Big)=0.
\end{equation}
\begin{remark}\label{remarkonrealanalyticityofVkrho} 
It easy to verify that $(\rho,\Lambda) \mapsto V_{k}^{\rho}(\Lambda)$ is real-analytic as a mapping $\mathcal{U} \to C^{\ell+\alpha}(\mathbb{T})$ for any $\ell \geq 0$.  Note also that due to Remark~\ref{rangeCremark} the left-hand side of \eqref{KinematicConditionlocalfinal} necessarily lies in $\mathring{C}^{\ell-1+ \alpha} (\mathbb{T})$.
\end{remark}

\subsection{Bernoulli condition}\label{bernoulli section}
The dynamic condition requires that the pressure along each hollow vortex boundary is constant. By Bernoulli's principle this is equivalent to the relative field having constant modulus on each  $\Gamma_k$, which in the conformal domain becomes
\begin{equation}\label{bernoulliconditionlocal1}
    \left|\frac{w_{\zeta}}{f_{\zeta}}-c \right|^2 =q_k^2\quad \text{on } \partial B_{\rho}(\zeta_k)
 \end{equation}
 for $k = 1, \ldots, M$. Recall that $q_1, \ldots, q_M \in \mathbb{R}$ are the Bernoulli constants.
 
 As discussed in Section~\ref{intro conformal section}, the relative velocity field for the point vortex system is unbounded in any neighborhood of a vortex center, and so to have a real-analytic operator equation that captures both the point vortex model and hollow vortex model, a normalization must be made. Specifically, multiplying the left-hand side of~\eqref{bernoulliconditionlocal1} by $\rho^2$ and evaluating at a point $\zeta = \zeta_k + \rho \tau$, we find that
\begin{align*}
    \rho^2\left|\frac{w_{\zeta}}{f_{\zeta}}-c \right|^2 
    &=\frac{\big|\rho(w_{\zeta}^0(\zeta_k+\rho\tau)+\mathcal{Z}_k^{\rho}\nu')-\rho c(1+\rho\mathcal{Z}_k^{\rho}\mu')\big|^2}{|1+\rho\mathcal{Z}_k^{\rho}\mu'|^2} \\
    & = \big|\rho(w_{\zeta}^0(\zeta_k+\rho\tau)+\mathcal{Z}_k^{\rho}\nu')-\rho c(1+\rho\mathcal{Z}_k^{\rho}\mu')\big|^2\big(  1-2\rho\realpart(\mathcal{C}^{0}\mu') \big) + O(\rho^2) \\      &=\frac{\gamma_k^2}{4\pi^2}+\rho \frac{\gamma_k^2}{2\pi^2}\realpart\left(\frac{2\pi i\tau}{\gamma_k}(\mathcal{V}_k+\mathcal{C}^{0}\nu ')-\mathcal{C}^{0}\mu' \right)  + O(\rho^2),
\end{align*} 
where all $O(\rho^2)$ terms are measured in $C^{\ell-1+\alpha}(\mathbb{T})$. We have made use of the fact that 
\[
	\mathcal{Z}_k^\rho = \mathcal{C}^0+O(\rho) \qquad \textrm{in }\mathcal{L}(\mathring{C}^{\ell-1+\alpha}(\mathbb{T},\mathbb{C}),C^{\ell-1+\alpha}(\mathbb{T},\mathbb{C}))
\]
thanks to Proposition \ref{firstZ_krhoassymprotic}, and so, for $(\rho,\mu)$ in a sufficiently small neighborhood of the origin in $\mathbb{R} \times C^{\ell+\alpha}(\mathbb{T})^M$, one has
\begin{equation*}
\frac{1}{| 1+\rho \mathcal{Z}_k^{\rho}\mu' |^2 } = 1-2\rho\realpart{\mathcal{C}^0\mu_k'} + O(\rho^2) \qquad \text{ in } C^{\ell-1+\alpha}(\mathbb{T}).
\end{equation*}
The above expansion justifies introducing a mapping 
\[
	\mathcal{B}^\rho = \left( \mathcal{B}_1^\rho(\mu,\nu,\Lambda), \, \ldots, \, \mathcal{B}_M^\rho(\mu,\nu,\Lambda) \right)
\]
defined on an appropriate open set (to be discussed in the next subsection) by
\begin{equation}\label{definition of calB_k}
 \frac{\gamma_k^2}{4\pi^2} + \rho \mathcal{B}_{k}^{\rho} \colonequals  \rho^2\Big|\frac{w_{\zeta}}{f_{\zeta}}-c \Big|^2,
\end{equation}
where $w$ and $f$ are viewed as determined by $(\mu,\nu, \Lambda)$ according to \eqref{ansatzforf} and \eqref{ansatzforw}.
Formally, from the above calculation, we have that 
\begin{equation} 
\mathcal{B}_{k}^{\rho} = \frac{\gamma_k^2}{2\pi^2}\realpart\Big(\frac{2\pi i\tau}{\gamma_k}(\mathcal{V}_k+\mathcal{C}^{0}\nu ')-\mathcal{C}^0\mu' \Big) + O(\rho) \quad \text{ in } C^{\ell-1+\alpha}(\mathbb{T}),
\end{equation}

In total, then, the dynamic condition on the $k$-th vortex boundary~\eqref{bernoulliconditionlocal1} can be expressed as 
\begin{equation}\label{bernulliconditionforB_k}
\mathcal{B}_k^{\rho} = Q_k,
\end{equation}
for a normalized Bernoulli constant $Q_k \in \mathbb{R}$ for each $k=1,\ldots , M$. Note that this agrees with the non-periodic case in~\cite{chen2023desingularization} as it depends only on the local structure of the velocity field near the vortex boundary. The effects of periodicity enter in at higher-order powers of $\rho$.

\subsection{Abstract operator equation}\label{Abstract operator equation Section}
With the kinematic and Bernoulli conditions reformulated in terms of the layer potentials in the previous subsections, the periodic hollow vortex problem can now be rephrased as an abstract operator equation fit for the implicit function theorem. 

Consider a non-degenerate periodic point vortex configuration $\Lambda_0$, that is, a solution 
\[
	\Lambda_0 = (\zeta_1^0,\ldots \zeta_M^0, \gamma_1^0,\ldots , \gamma_M^0 , c^0,T^0) \in \mathbb{P}
\]
to the algebraic system of equations \eqref{zero-set mathcal V}. We may then decompose $\mathbb{P} = \mathcal{P}\times \mathcal{P}'$ so that if $\Lambda = (\lambda, \lambda')$ then $D_{\lambda}\mathcal{V}$ is an isomorphism. Throughout the rest of the analysis, $\lambda'$ will be fixed to its initial value in $\Lambda_0$, while the parameters collected in $\lambda$ will be allowed to vary along the families of solutions to be constructed. 
 
Our new unknowns will be the densities $\mu$ and $\nu$, along with the normalized Bernoulli constants $Q = (Q_1, \ldots, Q_M)$. We therefore work with the space
$$
	\mathscr{W}\colonequals \mathring{C}^{\ell+\alpha}(\mathbb{T})^M\times \mathring{C}^{\ell+\alpha}(\mathbb{T})^M\times \mathbb{R}^M
$$
for an $\alpha \in (0,1)$ and $\ell\geq 1$ fixed. Slightly abusing notation, we will gather the unknowns along with the varying parameters $\lambda$ together into unknown $u$, which will then be an element of the space
\[
	u = (\mu,\nu,Q,\lambda) \in \Xspace \colonequals \Wspace \times \mathcal{P}.
\]

Now, consider the nonlinear operators $\mathscr{A} = (\mathscr{A}_1,\ldots , \mathscr{A}_M)$ and $\mathscr{B} = (\mathscr{B}_1,\ldots , \mathscr{B}_M)$ whose $k$-th component is given by
\begin{equation}\begin{aligned}\label{definitionsforscrA,B}
    \mathscr{A}_k(u; \rho) &= \realpart \Big(\tau\Big(\mathcal{V}_k^{\rho}+ \mathcal{Z}_k^{\rho}[\nu']-c\rho \mathcal{Z}_k^{\rho}[\mu'] \Big)\Big) ,\\
\mathscr{B}_{k}(u; \rho) &= \mathcal{B}_k^{\rho}(\mu,\nu,\lambda) - Q_k,
\end{aligned}\end{equation}
where $\mathcal{B}_{k}^{\rho}$ was defined in \eqref{definition of calB_k}.  The zero set of $\mathscr{A}$ and $\mathscr{B}$ correspond to the kinematic and Bernoulli conditions respectively. Therefore, the nonlocal formulation of the kinematic~\eqref{KinematicConditionlocalfinal} and Bernoulli~\eqref{bernulliconditionforB_k} conditions are equivalent to 
\be
\label{definition nonlocal formulation}
	\F(u; \rho) = 0
\ee
for the nonlinear mapping
\be
\label{definition F}
\mathscr{F}\colonequals (\mathscr{A},\mathscr{B}) \colon \mathcal{O} \subset \Xspace \times \mathbb{R} \to \Yspace,
\ee
with codomain 
\[
	\Yspace \colonequals \mathring{C}^{\ell-1+\alpha}(\mathbb{T})^M\times C^{\ell-1+\alpha}(\mathbb{T})^M
\]
and $\mathcal{O}$ an open neighborhood of the point vortex configuration. More precisely, $\mathcal{O}$ must be chosen so that (i) we have $(\lambda, \lambda^\prime,\rho) \in \mathcal{U}$, and (ii) the mapping $f$ given by \eqref{ansatzforf} is indeed conformal on the boundary. Through Lemma~\ref{complexlemmaconformality} and a continuity argument, conformality on the boundary will imply conformality on $\confD_\rho$ as well. Thus, we define
\begin{equation}\label{definition for calO}
\mathcal{O}_{\delta} = \{(\mu,\nu,Q,\lambda,\rho)\in \mathscr{W}\times \mathcal{U}_{\delta}\text{ : } \inf_{k}\inf_{\mathbb{T}}|1+\mathcal{Z}_k[\mu']| >0   \}, \qquad  \mathcal{O} \colonequals \bigcup_{\delta >0} \mathcal{O}_{\delta}.
\end{equation}

The next lemma summarizes the conclusions of this section. 

\begin{lemma}[Nonlocal formulation] \label{Solutionslemma}
Consider the mapping $\F$ given by~\eqref{definition F}.
\begin{enumerate}[label=\rm(\alph*)]
	\item \label{real-analytic part} $\F$ is real-analytic $\mathcal{O} \to \Yspace$. 
	\item \label{point vortex part} The zero-set of $\F(\placeholder; 0)$ corresponds to steady point vortex configurations in that $\F(\mu,\nu,Q,\lambda; 0) = 0$ if and only if $(\mu,\nu,Q) = (0,0,0)$ and $\mathcal{V}(\lambda,\lambda^\prime) = 0$.
	\item \label{equivalence part} Suppose  $u = (\mu,\nu,Q,\lambda, \rho) \in \F^{-1}(0)$ for $\rho \neq 0$ and define $f$ by~\eqref{ansatzforf}, $w$ by~\eqref{ansatzforw}, and $\Gamma_k \colonequals f(\partial B_\rho(\zeta_k))$. If $f$ satisfies the winding condition~\eqref{integralcondition1}, is injective on $\partial{\confD_\rho}$, and if none of the Jordan curves $\Gamma_1, \ldots, \Gamma_M$ encloses another, then $\mathsf{w} \colonequals w \circ f^{-1}$ solves the periodic hollow vortex problem \eqref{PHVP} on the domain $\fluidD \colonequals f(\confD_\rho)$.
\end{enumerate}
\end{lemma}
\begin{remark}
Observe that because $f$ is a perturbation of the identity that decays uniformly to $0$ at infinity, the winding condition, boundary injectivity, and topological condition on $\Gamma_1, \ldots, \Gamma_M$ asked for in part~\ref{equivalence part} hold in a neighborhood of $(0,0,0,\lambda_0,0)$ in $\Xspace \times \mathbb{R}$. We will use Lemma~\ref{complexlemmainjectivity} and a continuity argument when we construct non-perturbative solutions in Section~\ref{global section}.
\end{remark}
\begin{proof}
Part~\ref{real-analytic part} follows from the discussion in Section~\ref{Layer-potential representations} and Remark~\ref{remarkonrealanalyticityofVkrho}, while part~\ref{point vortex part} is an easy adaptation of \cite[Lemma 4.3]{chen2023desingularization}. 

Now let $u$ be given as in the statement of part~\ref{equivalence part}. From the definition of $\mathcal{O}$, $\partial_\zeta f$ is nonvanishing on $\partial \confD_\rho$, and hence by Lemmas~\ref{complexlemmaconformality} and~\ref{complexlemmainjectivity}, $f$ is conformal and injective on $\overline{\confD_\rho}$. Since $\nu \in \mathring{C}^{\ell+\alpha} (\mathbb{T})^M$ it is easily seen from the ansatz \eqref{ansatzforw} that $w \circ f^{-1}$ is holomorphic and hence \eqref{holomorphicityofvelocity} holds. The kinematic~\eqref{kinematicconditionintro} and dynamic boundary conditions~\eqref{bernoulliconditionintro} are satisfied by the construction of $\mathscr{F}$. Finally for the circulations \eqref{circulationconditionintro} a simple change of variables reveals that
\begin{align*}
\int_{\Gamma_k} \partial_{z} \mathsf{w} \, dz &= \int_{\Gamma_k}\frac{w_{\zeta}(f^{-1}(z))}{f_{\zeta}(f^{-1}(z))}\, dz 
=\int_{\partial B_{\rho}(\zeta_k)} \left( w_{\zeta}^{0}(\zeta) + \mathcal{Z}^{\rho}[\nu'] (\zeta) \right) \, d\zeta = \gamma_k + \int_{\mathbb{T}}\mathcal{Z}_{k}^{\rho}[\nu'](\tau) \rho \,  d\tau.
\end{align*}
Recall from \eqref{definition for Zkrho} that $\mathcal{Z}_k^{\rho}\nu' = \mathcal{C}^{\rho} \nu_k' +{h}_{k}$, where $h_k$ is holomorphic on $B_1$. Thus, 
$$
\int_{\mathbb{T}}\mathcal{Z}_{k}^{\rho}[\nu'](\tau) \rho \,  d\tau  =  \int_{\mathbb{T}}\mathcal{C}^{\rho}[\nu_k'] (\tau)\, \rho d\tau = 0,
$$
where the second equality holds due to Remark \ref{rangeCremark} and the assumption that $P_0 \nu = 0$. 
\end{proof}

\section{Vortex desingularization}
\label{local section}
In this section, we will state and prove a rigorous version of Theorem \ref{local desingularization theorem} on ``local'' vortex desingularization. Suppose that $\Lambda_0 = (\lambda_0,\lambda_0^\prime)$ is a non-degenerate point vortex configuration. Then 
\[
	u_0 = (0,0,0,\lambda_0,0) \in \mathcal{O}
\]
is in the zero-set of the nonlinear operator $\F$ from~\ref{definition F}. Our strategy is simply to apply the implicit function theorem to $\F$ at $u_0$ to obtain a curve of solutions that is parameterized by $\rho$ for $|\rho| \ll 1$. Thanks to Lemma~\ref{Solutionslemma}, each of these with $\rho \neq 0$ will correspond to a hollow vortex configuration. As $\F$ is real analytic, one can in principle compute these solutions to arbitrary order. We present here only the leading-order asymptotics. These closely resemble those found by Chen, Walsh, and Wheeler~\cite{chen2023desingularization} for the finite vortex case, except that the effective straining velocity field is takes a somewhat more complicated form.
\be
\label{definition Sk}
	S_k \colonequals -\frac{1}{2} \left (\frac{\gamma_k^0 \pi}{6iT^2} + \sum_{j\neq k}\frac{\gamma_j^0 \pi}{2iT^2}  \csc^2\left(\frac{\pi}{T}(\zeta_k^0-\zeta_j^0) \right) \right).
\ee
It is interesting to note that, formally fixing all of the other parameters and sending the period to infinity, we have that
\[
	S_k \to - \frac{1}{2} \sum_{j \neq k} \frac{\gamma_j^0}{2\pi i} \frac{1}{(\zeta_j^0-\zeta_k^0)^2}  \qquad \textrm{as } T \to \infty.
\]
The second term in the limit above is the coefficient for the straining field in the finite vortex case; see \cite[Theorem 6.1]{chen2023desingularization}. 

Then, the main result is the following.
\begin{theorem}[Local desingularization] \label{localtheorem}
Let $\Lambda_0 = (\lambda_0 , \lambda_0') $ be a non degenerate steady vortex configuration. Then there exists $\rho_1 > 0 $ and a curve $\mathscr{C}_{\loc} \subset \mathscr{X}\times \mathbb{R}$ of solutions to the periodic hollow vortex problem that admits the real-analytic parametrization
$$
\mathscr{C}_{\text{loc}} = \{ (u^{\rho},\rho) = (\mu^{\rho} , \nu^{\rho}, Q^{\rho} , \lambda^{\rho} , \rho) \in \mathcal{O} :  |\rho| < \rho_1 \} \subset \F^{-1}(0).
$$
and exhibits the following properties.
\begin{enumerate}[label=\rm(\alph*)]
\item \textup{(Asymptotics)} The solutions along $\cm_\loc$ have the leading-order form 
\be 
\label{local asymptotics}
\begin{aligned}
    \mu_k^{\rho}(\tau) &= \frac{16 \pi}{\gamma_k} \rho \realpart{(iS_k\tau)}  + O(\rho^2) & \qquad & \textrm{in } \mathring{C}^{\ell+\alpha} (\mathbb{T}) \\
    \nu_{k}^{\rho}(\tau)&=-2\rho \realpart{(S_k\tau^2)} + O(\rho^2) & \qquad & \textrm{in }  \mathring{C}^{\ell+\alpha} (\mathbb{T})\\
    Q^{\rho}&= O(\rho^2) \\
    \lambda^{\rho}&= \lambda_0+O(\rho^2).
 \end{aligned}   
 \ee
    In particular, the curve $\mathscr{C}_{\text{loc}}$ bifurcates from the point vortex configuration $\Lambda_0$ at $\rho = 0$. 
\item \label{local symmetry part} \textup{(Symmetry)} The parametrization is symmetric in the sense that
\begin{equation}\label{symmetrystatement}
 	(\mu^{\rho}(\tau), \nu^{\rho} (\tau), Q^\rho, \lambda^\rho) = (\mu^{-\rho}(-\tau), -\nu^{-\rho} (\tau), Q^{-\rho}, \lambda^{-\rho},-\rho) \qquad \textrm{for all } |\rho| < \rho_1.
\end{equation}
\item \label{physical part} \textup{(Physicality)} For each $0 < |\rho| < \rho_1$, the corresponding conformal mapping $f^\rho$ determined from $(u^\rho,\rho)$ via~\eqref{ansatzforf} is conformal and injective on $\overline{\confD_\rho}$ and none of the Jordan curves $\Gamma_k^\rho \colonequals f^\rho(\partial B_\rho(\zeta_k^\rho))$ for $k  = 1, \dots, M$ encloses any other.
\end{enumerate}
\end{theorem}

We start by computing the Fr\'echet derivatives of $\mathscr{F}$ at $(u_0,0)$:
$$
D_u\mathscr{F}(u_0,0)\dot{u} \equalscolon
\begin{bmatrix}
0 & \mathscr{A}_{\nu}^{0}&0&\mathscr{A}_{\lambda}^{0}\\
\mathscr{B}_{\mu}^{0}&\mathscr{B}_{\nu}^{0}&\mathscr{B}_{Q}^{0}&\mathscr{B}_{\lambda}^{0}
\end{bmatrix}
\begin{bmatrix}
    \dot{\mu}\\ \dot{\nu} \\ \dot{Q}\\ \dot{\lambda}
\end{bmatrix},
$$
where the lower indices indicate the variable that we are differentiating with respect to and the upper zeros correspond to evaluating at $(u^0,0)$. Note that these operators have block-diagonal form in that
\begin{equation}
\label{Frechet derivatives}
\begin{aligned}
    \begin{bmatrix} \mathscr{A}_{k \nu}^{0} & \mathscr{A}_{k \lambda}^{0} \end{bmatrix} \begin{bmatrix} \dot{\nu} \\ \dot \lambda \end{bmatrix} & =\realpart{(\tau\mathcal{C}^{0}\dot{\nu}_k')} + \realpart{\left( \tau \mathcal{V}_{k\lambda}^{0}\dot{\lambda}\right)} \\
     \begin{bmatrix} \mathscr{B}_{k\mu}^{0} & \B_{k\nu}^0 & \B_{kQ}^0 & \B_{k\lambda}^0  \end{bmatrix}
    \begin{bmatrix}
        \dot{\mu}\\ \dot{\nu}\\ \dot{Q} \\ \dot \lambda
    \end{bmatrix} &=
    \frac{(\gamma_k^{0})^2}{2\pi^2}\realpart{ \left( \frac{2\pi i}{\gamma_k^{0}}\tau\mathcal{C}^{0}\dot{\nu}_k'-\mathcal{C}^{0}\dot{\mu}_k'  \right)} \\
    	& \qquad -\dot{Q_k} + \frac{\gamma_k}{\pi}\realpart{\left(i\tau\mathcal{V}_{k\lambda}\dot{\lambda}\right)}.
\end{aligned}
\end{equation}
Here, we are using the convention that variations are adorned with a dot. As each of the linearized operators are Fourier multipliers with respect to $\mu$ and $\nu$, determining their kernel and range is quite simple. In particular, observe that for any $\varphi \in C^{\ell+ \alpha} (\mathbb{T})$ with Fourier series expansion 
$$
\varphi(\tau) = \sum_{m\in \mathbb{Z}} \hat{\varphi}_m \tau^m,
$$
then direct computation reveals that
\begin{equation}\label{operators}
\begin{aligned}
      \realpart\mathcal{C}^{0}[\varphi'](\tau)&=\sum_{m=-\infty}^{1}|m|\hat{\varphi}_m\tau^{m-1}+\sum_{m=1}^{\infty}|m|\hat{\varphi}_m\tau^{m+1} \\
    \realpart\big(\tau\mathcal{C}^{0}[\varphi'](\tau) \big)&=\frac{1}{2}\sum_{m\in\mathbb{Z}}|m|\hat{\varphi}_m\tau^m, \quad
   \realpart\big(i\tau\mathcal{C}^{0}[\varphi'](\tau) \big)=-\frac{1}{2}\sum_{m\in\mathbb{Z}}im\hat{\varphi}_m\tau^m.
     \end{aligned}
\end{equation}
From these formulas it is apparent that all the mappings in \eqref{operators} are injective when their domain is restricted to $\mathring{C}^{\ell+\alpha}$, and their range is: 
\[
\begin{aligned}
    \range{\realpart{\left(\tau \mathcal{C}^{0}\partial_{\tau} \big|_{\mathring{C}^{\ell+\alpha}(\mathbb{T})} \right)}} &= \mathring{C}^{\ell-1+\alpha}(\mathbb{T}), \qquad
    \range{\realpart{\left(\mathcal{C}^{0}\partial_{\tau}\big|_{\mathring{C}^{\ell+\alpha}(\mathbb{T})}  \right)}}= P_{>1} C^{\ell-1+\alpha}(\mathbb{T}),
\end{aligned}
\]
where $P_{\bullet}$ are the projection operators defined in Section \ref{notation}. Observe that $\mathscr{A}_{\nu}^{0}$ is a diagonal operator with $\realpart{(\tau \mathcal{C}^{0}\partial_{\tau})}$ along the diagonal, and $\realpart (\tau \mathcal{C}^{0}\partial_{\tau})$ is bijective $\mathring{C}^{\ell+\alpha} (\mathbb{T})\to\mathring{C}^{\ell-1+\alpha} (\mathbb{T})$. Thus $\mathscr{A}_{\nu}^{0}$ is bijective $\mathring{C}^{\ell+\alpha} (\mathbb{T})^M\to\mathring{C}^{\ell-1+\alpha} (\mathbb{T})^M$ and the inverse is likewise a diagonal operator with the diagonal entries given by \eqref{inverrseofAnu}. 

The next lemma applies these formulas to establish that the linearized hollow vortex problem at $(u_0,0)$ is Fredholm with Fredholm index determined entirely by the linearization of the point vortex system at $\Lambda_0$. The proof is essentially the same as in \cite[Lemma 6.3]{chen2023desingularization}, but we present it for completeness and later reference. 

\begin{lemma}[Kernel and range] \label{kernelrangelemma}
Suppose that $\Lambda_0=(\lambda_0,\lambda_0')$ represents a steady point vortex configuration and let $u_0:=(0,\lambda_0)\in\mathscr{X}$. Then the Fr\'echet derivative $D_u\mathscr{F}(u_0,0) \colon \mathscr{X}\rightarrow \mathscr{Y}$ is Fredholm with
\begin{align*}
    \dim\ker D_u\mathscr{F}(u^0,0)&=\dim\ker D_{\lambda}\mathcal{V} (\Lambda_0),\quad 
    \codim\range D_{u}\mathscr{F}(u^0,0)=\codim\range D_{\lambda}\mathcal{V}(\Lambda_0)
\end{align*}
\end{lemma}
\begin{proof}
    Inverting $\mathscr{A}_{\nu}^{0}$, we consider the row-reduced operator:
    \begin{equation}
        \mathscr{L}
        \begin{bmatrix}
            \dot{\mu}\\ \dot{Q} \\ \dot{\lambda}
        \end{bmatrix}:=\mathscr{B}_{\mu}^{0}\dot{\mu}+\mathscr{B}_{Q}^{0}\dot{Q}+(\mathscr{B}_{\lambda}^{0}-\mathscr{B}_{\nu}^{0}(\mathscr{A}_{\nu}^{0})^{-1}\mathscr{A}_{\lambda}^{0})\dot{\lambda}
.    \end{equation}

Clearly $D_u\mathscr{F}(u^0,0)$ is Fredholm if and only if $\mathscr{L}$ is Fredholm. Moreover the dimensions of $\ker{\mathscr{L}}$ and $\ker{D_u\mathscr{F}(u^0,0)}$ coincide, as do the codimensions of their range. Recalling~\eqref{Frechet derivatives} and \eqref{operators}, we derive the explicit formulas
\begin{align*}
(\mathscr{A}_{\nu}^{0})^{-1}\mathscr{A}_{\lambda}^{0}\dot{\lambda}&= \mathcal{V}_{\lambda}^{0}\dot{\lambda}\tau+\overline{\mathcal{V}_{\lambda}^{0}\dot{\lambda}} \frac{1}{\tau} = -2\realpart(\mathcal{V}_{\lambda}^0\dot{\lambda}\tau) \quad \text{ and }\quad 
\mathscr{B}_{\nu}^{0}(\mathscr{A}_{\nu}^{0})^{-1}\mathscr{A}_{k\lambda}^{0}\dot{\lambda}=-\frac{\gamma_k^0}{\pi}\realpart \Big(i\tau \mathcal{V}_{k\lambda}^{0} \Big),
\end{align*}
and hence
\begin{equation}\label{lambdaform}
\mathscr{L}_k
\begin{bmatrix}
\dot{\mu} \\ \dot{Q}\\ \dot{\lambda}
\end{bmatrix}
=-\frac{(\gamma_k^{0})^2}{2\pi^2}\realpart \mathcal{C}^{0}\dot{\mu}_k'-\dot{Q}_k+\frac{2\gamma_k^0}{\pi}\realpart \Big( i\tau\mathcal{V}_{k\lambda}\dot{\lambda} \Big).
\end{equation}
Now the dimension count becomes clear. Suppose that $(\dot\mu,\dot Q, \dot\lambda) \in\ker \mathscr{L}$. Projecting \eqref{lambdaform} to the subspace of constants we get that $\dot{Q}=0$. Likewise, applying the projection $P_1$, shows that $\dot{\lambda}\in\ker D_{\lambda}\mathcal{V}(\Lambda_0)$. Finally, applying $P_{>1}$, we conclude that in fact  $\dot{\mu}=0$. Therefore, the dimension of $\ker D_u\mathscr{F}(u^0,0)$ agrees with that of $\ker D_{\lambda}\mathcal{V}(\Lambda_0)$. Similarly we characterize the range of $\mathscr{L}$. Indeed, we see that for $\varphi\in {C}^{\ell-1+\alpha}(\mathbb{T})^M$
\begin{equation}
    \varphi\in \range\mathscr{L} \qquad \text{ if and only if } \qquad \int_{\mathbb{T}}\varphi\overline{\tau} \, d\theta\in\range D_{\lambda}\mathcal{V} (\Lambda_0)
.\end{equation}
Thus the codimensions of $\range D_u\mathscr{F}(u^0,0)$ and $\range D_{\lambda}\mathcal{V}(\Lambda_0)$ coincide.
\end{proof}
Now all the tools are in place to prove the local desingularization theorem. Again, the main contours of the argument follow the finitely many vortex case in \cite{chen2023desingularization}.

\begin{proof}[Proof of Theorem~\ref{localtheorem}]
Thanks to Lemma~\ref{Solutionslemma}, we know that $\F(u_0,0) = 0$. The non-degeneracy assumption means that $D_\lambda \mathcal{V}(\Lambda_0)$ is an isomorphism, and hence by  Lemma~\ref{kernelrangelemma} the linearized operator $D_u \F(u_0,0)$ is likewise invertible. The existence of a curve $\cm_\loc \subset \F^{-1}(0)$ therefore follows from the (real-analytic) implicit function theorem. We can compute the asymptotic formulas~\eqref{local asymptotics} exactly as in \cite[Theorem 6.1]{chen2023desingularization}, and so we omit the proof. The symmetry of the parameterization claimed in part~\ref{local symmetry part} is a consequence of uniqueness and the observation that 
\[
	\F(\mu,\nu,Q,\lambda; \rho) = 0 \qquad \textrm{if and only if} \qquad \F(\mu(-\placeholder), -\nu(-\placeholder), Q, \lambda; -\rho) = 0.
\] 

Finally, to see that for $|\rho| > 0$, these are indeed solutions of the hollow vortex problem~\eqref{PHVP}, it is necessary to further confirm that 
\[
	f^\rho = \id + \rho^2 \mathcal{Z}^\rho(\Lambda^\rho)[\mu^\rho] \equalscolon \id + g^\rho
\]
is conformal and injective on $\overline{\confD_\rho}$. Recall from the discussion in Section~\ref{Layer-potential representations} that 
\[
	\partial_\zeta g^\rho(\zeta), \, \partial_\zeta^2 g^\rho(\zeta) \to 0 \quad \textrm{as } \imagpart{\zeta} \to \pm\infty
\]
and this convergence is uniform in $\rho$ for compact subsets of $(-\rho_1, \rho_1)$. By the boundedness of $\mathcal{Z}^\rho$ and the asymptotics~\eqref{local asymptotics}, it follows that
\begin{align*}
	f^\rho & = \id + \rho^2 \mathcal{Z}^\rho(\Lambda^\rho)[\mu^\rho] = \id + O(\rho^3) & \qquad & \textrm{in } C^{\ell+\alpha}(\overline{\confD_\rho}) \\
	\partial_\zeta f^\rho & = 1 + \rho \mathcal{Z}^\rho(\Lambda^\rho)[\partial_\tau \mu^\rho] = 1 + O(\rho^2) & \qquad & \textrm{in } C^{\ell-1+\alpha}(\overline{\confD_\rho}).
\end{align*}
Thus, 
\[
	\frac{1}{2\pi i } \int_{\partial B_{\rho}(\zeta_k)}\frac{\partial_{\zeta}^2f}{\partial_{\zeta}f} \, d\zeta =  O (\rho).
\]
Since the winding number only takes integer values, the above integral must be zero for small enough. Perhaps by shrinking $\rho_1$, we can moreover ensure the vortex boundaries do not enclose on another. The statement now follows from Lemmas~\ref{complexlemmaconformality} and \ref{complexlemmainjectivity}.
\end{proof}

\section{Global continuation}\label{global section}

Up to this point, we have shown that a non-degenerate point vortex configuration  $\Lambda_0$  can be desingularized to give a local curve $\cm_\loc$ of solutions to the hollow vortex problem. In this section, we continue $\cm_\loc$ globally and prove Theorem~\ref{global desingularization theorem}. Some of the analysis is quite similar to that in the finite vortex case~\cite{chen2023desingularization}. We will omit some of these arguments in order to highlight the novel elements in the periodic case. 

Since $\F$ is real analytic and, we have shown that $D_u \F$ is an isomorphism at the point vortex configuration if it is non-degenerate, the existence of a (maximal) global curve extending $\cm_\loc$ follows from the analytic global bifurcation theory of Dancer~\cite{dancer1973bifurcation,dancer1973globalstructure} and Buffoni--Toland~\cite{buffoni2003analytic}. The result is the following. 

\begin{theorem}[Preliminary global continuation] \label{prelim global theorem} 
In the setting of Theorem~\ref{localtheorem}, there exists a curve $\cm$ that admits the global $C^0$ parameterization  
 \[ \cm \colonequals  \left\{ (u(s), \rho(s)) : s \in \mathbb{R}  \right \} \subset \F^{-1}(0),\]
extends $\cm_\loc$ in that $\cm_\loc \subset \cm$, and satisfies the following.  
  \begin{enumerate}[label=\rm(\alph*)]
  \item \label{K well behaved} At each $s \in \mathbb{R}$, the linearized operator $\F_u(u(s); \rho(s)) \colon \Xspace \to \Yspace$ is Fredholm index $0$.
  \item \label{K alternatives} One of the following alternatives holds as $s \to \infty$ and $s \to -\infty$.
    \begin{enumerate}[label=\rm(A\arabic*)]
    \item  \label{K blowup alternative}
      \textup{(Blowup)}  The quantity 
      \begin{align}
        \label{K blowup}
         N(s) \colonequals  \n{u(s)}_{\Xspace} + |\rho(s)| +\frac 1{\dist((u(s),\param(s)), \, \dell \mathcal{O})} \longrightarrow \infty.
      \end{align}
    \item \label{K loss of compactness alternative} \textup{(Loss of compactness)} There exists a sequence $s_n \to \pm\infty$ for $N(s_n)$ is uniformly bounded, but $( u(s_n), \rho(s_n) )$ has no convergent subsequence in $\Xspace \times \mathbb{R}$.  
      \item \label{K loss of fredholmness alternative} \textup{(Loss of Fredholmness)}    There exists a sequence $s_n \to \pm\infty$
        for $N(s_n)$ is uniformly bounded and so that $(u(s_n), \rho(s_n)) \to (u_*, \rho_*) \in \Xspace \times \mathbb{R}$ in $\Xspace\times \mathbb{R}$, however $\F_u(u_*;\rho_*)$ is not semi-Fredholm.  
   \item \label{K loop} \textup{(Closed loop)} There exists $P > 0$ such that $(u(s+P), \rho(s+P)) = (u(s), \rho(s))$ for all $s \in (0,\infty)$. 
        \end{enumerate}
          \item \label{K reparam} Near each point $(u(s_0),\rho(s_0)) \in \mathscr{K}$, we can locally reparametrize $\cm$ so that $s\mapsto (u(s),\rho(s))$ is real analytic.
 \item \label{K maximal part} The curve $\cm$ is maximal in the sense that, if $\mathscr{K} \subset \F^{-1}(0)$ is a locally real-analytic curve containing $(u_0,0)$ and along which  $\F_u$ is Fredholm index $0$, then $\mathscr{K} \subset \cm$. 
  \end{enumerate}
\end{theorem}
\begin{proof}
This is an immediate consequence of~\cite[Theorem B.1]{chen2024global}, which is a global implicit function theorem rephrasing of the classical Dancer, Toland--Buffoni theory.
\end{proof}

The set of alternatives  \ref{K blowup alternative}--\ref{K loop} above are not obviously physically meaningful. The goal of this section is thus to refine them down to just conformal degeneracy or velocity field degeneracy. Recall that these refer to the blowup of the quantities $\Nconf$ from \eqref{Ncdefinition} and $\Nvel$ from \eqref{Nvdefinition}, respectively.  

A first step in analyzing the alternatives \ref{K blowup alternative}--\ref{K loop} is to obtain a Schauder type estimate for the linearized operators $D_{u}\mathscr{F}(u;\rho)$ at any solution $(u,\rho) \in \F^{-1}(0)$. As usual, these will help us to exclude the loss of Fredholmness alternative~\ref{K loss of fredholmness alternative}. This can be done essentially verbatim as in the finitely many vortex case~\cite[Lemma 7.1]{chen2023desingularization}. 

\begin{lemma}[Linear bounds] \label{SchauderLemma}
Let $(u,\rho) = (\mu,\nu,Q,\lambda ,\rho) \in \mathcal{O}_{\delta}$ with $\rho>0$ and suppose that $\inf_{\partial{\mathcal{D}_{\rho}}} |U| > 0$. Then linearized operator $D_{(\mu,\nu)} \mathscr{F}(u,\rho)$ enjoys the Schauder type estimate
\begin{equation}
    \| (\dot{\mu},\dot{\nu})\|_{C^{\ell + \alpha}(\mathbb{T})^{2M}} \leq C\big( \|D_{(\mu,\nu)} \mathscr{F} (u,\rho)[(\dot{\mu}, \dot{\nu})] \|_{C^{\ell -1 +\alpha}(\mathbb{T})^{2M}} +\|(\dot{\mu},\dot{\nu})\|_{C^{0}(\mathbb{T})^{2M}}\big) ,
\end{equation}
where the constant $C$ depends only on $\delta,\rho, \inf_{\partial \mathcal{D_{\rho}}}|U|$ and upper bounds for $\rho, Q ,\lambda$ and $\| (\mu,\nu)\|_{C^{\ell + \alpha}}$.
\end{lemma}

Recall that the winding condition~\eqref{integralcondition1} is a hypothesis for both Lemmas~\ref{complexlemmaconformality} and \ref{complexlemmainjectivity}, which we will eventually use to prove that the solutions along $\cm$ are indeed physical in that $f$ is conformal and injective on $\overline{\confD_\rho}$. In the proof of Theorem~\ref{localtheorem}, we confirmed that it  holds along the local curve $\cm_\loc$. By continuity and Lemma~\ref{complexlemmaconformality}, therefore, we have the following global statement.

\begin{lemma}[Conformality] \label{complexlemma71}
Let $(u,\rho) = (\mu,\nu,Q,\lambda , \rho)$ be given that is in the connected component of $\mathscr{F}^{-1}(0)\cap \mathcal{O}$, containing the point vortex solution $(u^0,0)$, then
\begin{equation}\label{intcondcomplex}
    \frac{1}{2\pi i} \int_{\partial \mathcal{D}_{\rho}} \frac{\partial_{\zeta}^2f}{\partial_{\zeta} f} \, d\zeta =0,
\end{equation}
and $\partial_{\zeta}f\neq 0$ in $\overline{\mathcal{D}}_{\rho}$.
\end{lemma}

Next, we seek to better characterize the blowup alternative~\ref{K blowup alternative} while also ruling out the loss of compactness alternative~\ref{K loss of compactness alternative}. The key ingredient is uniform $C^{\ell+\alpha}$ bounds on the densities $(\mu,\nu)$ in terms of bounds on $\Nconf$ and $\Nvel$. Define the relative stream function by
\begin{equation*}
    \Psi := \imagpart W - m_1
\end{equation*}
where the mass fluxes $m_1, \ldots, m_M$  are the constants introduced in \eqref{kinematicconditionintro}. Here, we have normalized the stream function to vanish on $\partial B_\rho(\zeta_1)$, but this does not influence the velocity field. Indeed, the relative velocity is $U = 2i \partial_{z}\Psi$. 

By definition, $\Psi $ satisfies the elliptic PDE
\begin{equation}
    \begin{cases}
        \Delta \Psi = 0 & \text{in } \mathcal{D}_{\rho} \\
        \Psi = m_k-m_1 & \text{on } \partial B_{\rho}(\zeta_k).
    \end{cases}
\end{equation}
We will obtain $C^{\ell +\alpha}$ control for $f$ by exploiting the relation 
$$
U\partial_{\zeta} f = 2i \partial_{\zeta}\Psi,
$$ 
along with Schauder-type estimates for $\Psi$. Recall the boundary traces of $|U|$ and $1/|U|$ are already controlled by $\Nvel$. By periodicity it suffices to work in a fundamental strip, so we revisit the domain
$$
\Omega_n = \{\zeta \in \mathbb{C}:\, 0< \realpart \zeta< T,\, -n<  \imagpart \zeta < n \},
$$
as in the proof of Lemma \ref{complexlemmaconformality}. In view of Remark \ref{Inegrationstrip}, we may assume that $\bigcup_{k=1}^{M}B_{\rho}(\zeta_k) \subset \Omega_n$ for $n = n(\zeta_1,\ldots , \zeta_M)$ sufficiently large; this assumption on $\Omega_n$ persists for the remainder of this subsection. Recall that the boundary of $\Omega_n$ consists of four line segments, which we denote by $\mathpzc{T}_{\,n},\mathpzc{R}_{\,\,n},\mathpzc{B}_{\,n}$ and $\mathpzc{L}_{\,n}$ for top, right, bottom and left parts of $\partial \Omega_n$ respectively. 
\begin{lemma}\label{boundsforU}
Let $(u,\rho) = (\mu,\nu , Q,\lambda,\rho) \in \mathscr{F}^{-1}(0)\cap \mathcal{O}_{\delta}$ with $\rho\neq 0$. Then for any $n>0$, sufficiently large we have 
$$
\|{U}\|_{L^{\infty}(\mathcal{D}_{\rho}\cap \Omega_n)} \leq C,
$$
where $C>0$ depends on upper bounds for $|\rho|, |q|, |c|$ and $|\lambda|$.
\end{lemma}
\begin{proof}
By definition, 
\[
	|U| \leq \left| \frac{\partial_\zeta w}{\partial_\zeta f} \right | + |c|,
\]
where $f$ and $w$ are constructed from $(\mu,\nu, \lambda, \rho)$ via~\eqref{ansatzforf} and \eqref{ansatzforw}.
The maximum modulus theorem and periodicity imply that $|U|$ is maximized on $\partial \confD_\rho \cup \mathpzc{T}_{\,n}\cup \mathpzc{B}_{\,n}$. The limiting behavior of $f_{\zeta},w_{\zeta}$, which follows from the limiting behavior of $\mathcal{Z}^{\rho}\mu^\prime$ and $\mathcal{Z}^\rho \nu^\prime$ established in Proposition \ref{proof of communication with tau derivative}, demonstrates that
\begin{equation*}
    |\partial_\zeta w(\zeta)| \to \left| \sum_k \frac{\gamma_k}{2\pi} \right|, \quad \partial_\zeta f(\zeta) \to 1 \qquad \textrm{as } \imagpart{\zeta} \to \pm\infty,
\end{equation*}
thus as ,
\begin{align*}
    \max_{\Omega_n\cap\mathcal{D}_{\rho}}|U|& \leq \max{\left\{ \max_{\confD_\rho} \left|\frac{w_{\zeta}}{f_{\zeta}}\right|,\, \left |\sum_k \frac{\gamma_k}{2\pi} \right|\right\}} + |c| \leq \max{ \left\{ |q| , \left|\sum_k \frac{\gamma_k}{2\pi} \right| \right\} } + 2|c|. \qedhere
\end{align*}
\end{proof}

Arguing as in \cite[Lemma 7.4]{chen2023desingularization}, we then get as a corollary the following. 
\begin{lemma}\label{lemma74}
Let $(u,\rho) = (\mu,\nu , Q,\lambda,\rho) \in \mathscr{F}^{-1}(0)\cap \mathcal{O}_{\delta}$ with $\rho \neq 0$, $q_k \neq 0 $ for all $k=1,\ldots , M $. Then for any $n>0$, sufficiently large we have that: 
$$
\|f\|_{C^{1+\alpha}(\mathcal{D}_{\rho}\cap \Omega_n)} \leq C 
$$
for a constant $C>0$ that depends on the lower bounds for $|\rho|,\delta $ and upper bounds for $|\rho|, n,N_{\textup{v}}, |c|, \lambda$ and $\|f\|_{C^{1}(\mathcal{D}_{\rho}\cap \Omega_n)}$ 
\end{lemma}

Now, we are ready to prove the main estimate of this section, which provides $C^{\ell+\alpha}$ control of the densities in terms of $N_{\text{v}}$ and $N_{\text{c}}$. Here, the proof deviates from the finite vortex case because of the more complicated limiting behavior of the conformal mapping and the velocity field at infinity. 

\begin{lemma}\label{uniformboundsformunu}
Let $(u,\rho) = (\mu,\nu , Q,\lambda,\rho) \in \mathscr{F}^{-1}(0)\cap \mathcal{O}_{\delta}$ with $\rho \neq 0$, $q_k \neq 0 $ for all $k=1,\ldots , M $. Then
\begin{equation*}
    \|\mu\|_{C^{\ell+\alpha}(\mathbb{T})^M} + \|\nu\|_{C^{\ell + \alpha}(\mathbb{T})^M} \leq C 
\end{equation*}
for a constant $C>0 $ depending on lower bounds for $|\rho|,\delta $ and upper bounds for $|\rho| , |\lambda|, N_{\textup{v}}$ and $\|\partial_{\zeta}f \|_{L^{\infty}(\partial\mathcal{D}_{\rho})}$.
\end{lemma}

\begin{proof}
To begin, let $\ell =1$. For this proof we fix $ n_0\colonequals n_0(\lambda) = 2\max\{|\zeta_1|,\ldots,|\zeta_M| \} +2|\rho|$; this will reduce the $n$ dependency of the constant $C$ of Lemma \ref{lemma74} to $\lambda$ and $\rho$ dependency. Our strategy consists of using the inverse bound \eqref{inversionforZk} for the operator $\mathcal{Z}_{k}^{\rho}$ to relate estimates for $\mu$ and $\mu'$ with estimates for $f$ and $\partial_{\zeta}f$. First we derive a bound for $\|\mu \|_{L^{\infty}(\mathbb{T})^M}$. Recall that 
$$
\mathcal{Z}_k^{\rho}[\mu'](\tau) = \frac{\partial_{\zeta}f(\zeta_k+\rho \tau)-1}{\rho} \in C^{\alpha}(\mathbb{T}).
$$
Morrey's inequality and \eqref{inversionforZk} lead to the upper bound, 
\begin{equation}\label{lemma65quickuse1}
    [\mu_{k}]_{\alpha; \mathbb{T}} \lesssim  \|\mu_k'\|_{L^{1/(1-\alpha)}(\mathbb{T})} \lesssim \|\mathcal{Z}_{k}^{\rho}\mu'\|_{L^{1/(1-\alpha)}(\mathbb{T})} \leq  C \big(\|\partial_{\zeta} f\|_{L^{\infty}(\partial \mathcal{D}_{\rho})} +1\big),
\end{equation}
where $C>0$ depends on the lower bounds for $\delta,|\rho|$ and the upper bounds for $\rho$. Moreover, since $P_0 \mu_k =0 $, we have that: \begin{equation*}
    \mu_k(\tau) = \frac{1}{2\pi i }\int_{\mathbb{T}} \frac{\mu_k(\tau)}{\sigma} d\sigma = \frac{1}{2\pi i }\int_{\mathbb{T}} \frac{\mu_k(\tau) - \mu_{k}(\sigma)}{\sigma} \, d\sigma.
\end{equation*}
By \eqref{lemma65quickuse1} it follows that,
\begin{equation}\label{analogus721}
    \|\mu \|_{L^{\infty}(\mathbb{T})^M} \lesssim \sum_{k} [\mu_k]_{\alpha; \mathbb{T}} \leq C.
\end{equation}
With the $L^{\infty}$ bound established for $\mu$, we turn our focus to $C^{\alpha}$ bounds for $\mu'$. Again we appeal to \eqref{inversionforZk} to deduce that
\begin{equation*}
    \| \mu ' \|_{C^{\alpha}(\mathbb{T})^M} \leq \sum_{k=1}^{M}\|\mathcal{Z}_k^{\rho} \mu' \|_{C^{\alpha}(\mathbb{T})}  =\frac{1}{\rho} \sum_{k=1}^{M}\| \partial_{\zeta}f(\zeta_k + \rho\, \cdot) -1\|_{C^{\alpha}(\mathbb{T})} \leq C (\|f \|_{C^{1+\alpha}(\overline{\mathcal{D}_{\rho}\cap \Omega_{n_0}})} +1)  ,
\end{equation*}
where $C$ depends on the allowable parameters. Bounding the right hand side of the previous inequality in terms of allowable parameters yields the desired bound for $\|\mu\|_{C^{1 + \alpha}}$. This would directly follow from Lemma \ref{lemma74}, however  the constant bounding $\| f \|_{C^{1+\alpha}(\mathcal{D}_{\rho}\cap \Omega_{n_0})}$ depends on $\|f \|_{C^{1}(\mathcal{D}_{\rho}\cap\Omega_{n_0})}$, which we do not control in terms of the allowable parameters. To overcome this obstacle we make the following two observations. First,
\begin{equation*}
    \|f-\id\|_{ L^{\infty}(\partial\mathcal{D}_{\rho})} \leq\sum_{k=1}^{M} \rho^2 \|\mathcal{Z}_{k}^{\rho}\mu\|_{L^{\infty}(\mathbb{T})} \lesssim C\|\partial_{\zeta}f\|_{L^{\infty}(\partial \mathcal{D}_{\rho})}.
\end{equation*}
Second that the function $f-\id$ is holomorphic, $T$-periodic, and exhibits the limiting behavior
\begin{equation*}
    \lim_{\imagpart{\zeta} \to \pm \infty} |f(\zeta)-\zeta| = \frac{1}{2\pi } \left|  \sum_{k} \widehat{\mu}_{k,-1} \right|, \quad \text{ where, } \quad \mu_k(\tau) = \sum_{m \in \mathbb{Z}} \widehat\mu_{k,m} \tau^m\
\end{equation*}
Thus
\begin{equation}\label{anag7.21}
|\hat{\mu}_{k,-1}|  \lesssim \|\mu\|_{L^{\infty} (\mathbb{T})^M} \lesssim C.
\end{equation}
Applying the maximum modulus theorem to $f-\id$, taking into consideration that the maximum cannot be attained on $\mathpzc{L}_{\,n}\cup\mathpzc{R}_{\,\,n}$ and our previous observations, reveals that for any $n\geq n_0$
\begin{equation*}
    \|f-\id\|_{C^{1}(\mathcal{D}_{\rho}\cap \Omega_{n_0})} \leq \|f-\id\|_{C^{1}(\mathcal{D}_{\rho}\cap \Omega_{n})}\lesssim C,
\end{equation*}
which as discussed earlier concludes the proof for the $C^{1+\alpha}$ bound for $\mu$. 

To bound $\nu$ in $C^{\ell + \alpha}$, we rearrange the terms from the kinematic condition we get
\begin{equation*}
    \realpart (\tau \mathcal{Z}_k^{\rho}\nu')=-\realpart \Big(\tau\Big(\mathcal{V}_k^{\rho}-c\rho \mathcal{Z}_k^{\rho}\mu' \Big)\Big):=g_k.
\end{equation*}
The bounds on $\mu $ ensure that $\|g_k\|_{C^{\alpha}} \leq C $. Lemma~\ref{inversionforZk} once more furnishes the same control for $\|\nu\|_{C^{1+\alpha}}$, whence
\begin{equation*}
   \|\mu\|_{C^{1+\alpha }(\mathbb{T})^M} + \|\nu\|_{C^{1+\alpha} (\mathbb{T})^M} \leq C.
\end{equation*}
A straightforward bootstrapping argument similar to \cite{chen2023desingularization}, gives us the desired $C^{\ell+\alpha} (\mathbb{T})$ control when $\ell \geq 2$. 
\end{proof}

\subsection{Proof of the global desingularization theorem}

Now we have all the tools required to prove the global continuation theorem.

\begin{proof}[Proof of Theorem~\ref{global desingularization theorem}]
Let $\cm$ be the global curve given by Theorem~\ref{prelim global theorem}. Recall, that $\cm$ is locally real analytic and that as $s\to \pm \infty$ one of the alternatives \ref{K blowup alternative}--\ref{K loop} must occur. In particular, the quantity 
\begin{align}\label{analogus722}
N(s) : = \|{\mu(s)}\|_{C^{\ell + \alpha}(\mathbb{T})^M} +  \|\nu(s)\|_{C^{\ell + \alpha}(\mathbb{T})^M} + |Q(s)| + |\lambda(s)| +|\rho(s)|+ \frac{1}{\dist  ((u(s),\rho(s)), \partial\mathcal{O} )}
\end{align}
is finite for all $s\in \mathbb{R}$. Let $\mathscr{C}_+$ denote the portion of $\mathscr{C}$ corresponding to $s> 0$. Without loss of generality we may assume that there exists $s_1 >0$ so that $\rho(s) >0 $ on $(0,s_1]$. Seeking a contradiction, assume that, 
\begin{equation}\label{blowupneeded}
    \sup_{s\geq s_1} (N_{\text{c}}(s) + N_{\text{v}}(s) + |\lambda(s)|) < \infty.
\end{equation}
We will prove that this necessarily excludes all the alternatives \ref{K blowup alternative}--\ref{K loop} of Theorem \ref{prelim global theorem}. Let $\mathcal{D}(s)$ and $f(s)$ denote the conformal domain and map respectively for the parameter value $s$, and $\fluidD(s) \colonequals f(s)(\confD(s))$.

The definition of $\mathcal{O}$ guarantees that $\mathcal{D} (s)$ is a circular domain in that none of the boundary circles intersects one another. Since the curve $\mathscr{C}$ contains the trivial point vortex configuration $(u_0,0)$ it follows from Lemma \ref{complexlemma71} that $\partial_{\zeta} f(s)$ is non vanishing on $\overline{\mathcal{D}(s)}$ and limits to $1$ as $\imagpart \zeta \to \pm\infty$. The maximum modulus principle then implies that, 
\begin{equation*}
    \sup_{s\geq s_1} \sup_{\mathcal{D}(s)} \left(|\partial_{\zeta }f(s)| + \frac{1}{|\partial_{\zeta}f(s)|} \right) \leq \sup_{s\geq s_1} \Nconf(s) <\infty .
\end{equation*}
Thus, distances in the conformal domain $\mathcal{D}(s)$ and the physical domain are uniformly comparable. Furthermore, finiteness of $\Nconf(s)$ and a simple continuity argument ensures that the boundary of $\fluidD(s)$ consists of $M$ Jordan curves per period, and these are uniformly separated in terms of $\Nconf(s)$ and do not enclose one another. Moreover, finiteness of $\Nconf(s)$ in particular implies that $f(s)$ is injective on $\partial \confD(s)$, and so it is globally injective according to Lemma~\ref{complexlemmainjectivity}. These considerations show that there exists  $\delta >0$ be given so that $\mathscr{C}_{+}\subset \mathscr{F}^{-1}(0)\cap \mathcal{O}_{\delta}$.

Recall that one of the circulations, say $\gamma_1$, is fixed along $\mathscr{C}^{+}$. Then, because the kinematic condition requires that the relative velocity field $U(s)$ is purely tangential along the vortex boundaries, 
\begin{align*}
    q_1(s) = \frac{1}{|\Gamma_1(s)|} \left| \int_{\Gamma_1(s)} U(s)\,  dz \right| = \frac{
    |\gamma_1|}{|\Gamma_1(s)|}.
\end{align*}
But,
\[
	|\Gamma_1(s)| = \int_{\Gamma_1(s)}|dz| = \int_{f(\partial_{\rho(s)}(\zeta_k))}|dz| \eqsim_{\|{\partial_{\zeta}f(s)}\|_{\infty}} \rho(s)
\]
and thus
\begin{align}\label{goodeinequality}
    q_1(s) \gtrsim \frac{\gamma_1}{\rho(s)}\Nconf(s).
\end{align}
On the other hand, $q_j(s) \leq \Nvel(s)$ for $j=1, \ldots, M$ simply by definition. Our assumption~\eqref{blowupneeded} that $N_\text{v}(s) < \infty$ together with~\eqref{goodeinequality} therefore provides a uniform lower bound on $|\rho(s )|$ along $\mathscr{C}_+$. Since $|\lambda(s)|$ is bounded, the magnitude of vortex centers $\zeta_k(s)$ is uniformly bounded, which implies a uniform upper bound on $\rho (s )$.

The previous paragraph shows that the last three terms in \eqref{analogus722} are controlled uniformly. Also we can easily see that by definition, 
$$
-\frac{\gamma_k(s)^2}{4\pi^2\rho(s)} < Q_k(s) < \rho(s)q_k(s)^2 < \rho(s) \Nvel(s)^2.
$$
Both the quantities on the far left and far right are uniformly controlled, thus $|Q(s)|$ is also bounded from above. To finally eliminate \ref{K blowup alternative} we must ensure finiteness of the first two terms in \eqref{analogus722}. Notice that boundedness of each $|q_k|$ follows from finiteness of $\Nvel(s)$. Then, the lower bound on $\rho(s)$ allows us to apply Lemma~\ref{uniformboundsformunu}, which provides uniform bounds on $\mu(s)$ and $\nu(s)$ in $C^{\ell+\alpha}$. In total, then, we have ruled out the blowup alternative~\ref{K blowup alternative}.

Notice that Lemmas~\ref{SchauderLemma} and \ref{uniformboundsformunu} directly exclude the loss of Fredholmness~\ref{K loss of fredholmness alternative} and loss of compactness alternatives \ref{K loss of compactness alternative}, respectively; their applicability follows from the bounds we obtained on the various parameters based on the assumption that \eqref{blowupneeded} holds. Finally, suppose that the closed loop alternative~\ref{K loop} occurs. But  this would mean that $\rho(P) =0 $ for some $P >0 $, which contradict the uniform lower bound on $\rho(s)$ we obtained above. Thus, the assumption in~\eqref{blowupneeded} cannot hold, which concludes the proof of the theorem.
\end{proof}

\section{Applications}\label{applications}

We next consider the three specific periodic vortex configurations discussed in the introduction: a von Kármán vortex street, a translating vortex array, and a translating $2P$ configuration. In fact, all of these are non-degenerate so one could apply Theorem~\ref{global desingularization theorem} directly. However, by taking advantage of additional symmetries, it is possible to fix more of the parameters along the global curve.

\subsection{Wave speed bound}
In this section, we establish a bound on the wave speed for translating periodic configurations. This result is of independent interest, but in particular allows us to exclude blowup of the parameters along the families of von Kármán hollow vortex streets and hollow vortex arrays.  

\begin{proposition}[Wave speed bound]\label{wavespeedProp}
For any translating (or stationary) periodic hollow vortex configuration the wave speed $c$, and the circulations $\gamma_k$, satisfy the following bounds
\begin{align}
    \label{circulationwavespeedbounds} |\gamma_k| \leq \rho  N_{\textup{c}}N_{\textup{v}}, \quad |c| \leq C  N_{\textup{v}} (1 +N_{\textup{c}}^2 )
\end{align}
where the constants $C$ depends on $M$, lower bounds of $\rho$ and upper bounds on $T$.
\end{proposition}
\begin{proof}
Recall that $U = w_{\zeta}/f_{\zeta} - c $. According to~\eqref{circulationconditionintro}, the circulation about the $k$-th vortex core is given by
\[
	2\pi i \gamma_k = \int_{\Gamma_k} \partial_{z} \mathsf{w}\, dz = \int_{\Gamma_k}U \circ f^{-1}\, dz + \int_{\Gamma_k}c\, dz= \int_{\Gamma_k}U \circ f^{-1}\, dz.
\]
Thus, 
\begin{equation}\label{circulationbound}
    |\gamma_k| \leq\frac{1}{2\pi} |\Gamma_k|N_{\text{v}}\leq \rho N_{\text{v}}N_{\text{c}},
\end{equation}
which is the circulation bound of~\eqref{circulationwavespeedbounds}. Now we proceed with the wave speed bound. In view of Remark~\ref{Inegrationstrip}, suppose that the balls $B_{\rho}(\zeta_k)$ are compactly contained in the strip of horizontal length $T$. Let $n$ be large enough such that
\[
\bigcup_{j=1}^{M}B_{\rho}(\zeta_k) \subset \Omega_n \colonequals \{x+yi \text{ : } 0< x <T,\, |y|\leq n \}.
\]
Then we have 
\be
\label{first wave speed bound}
\begin{aligned}
   2T  i c & = \int_{\partial \Omega_n}c\cot{\left( \frac{\pi}{T}(\zeta-\zeta_1)\right)}\, d\zeta = \int_{\partial \Omega_n}\left(U - \frac{\partial_\zeta w}{\partial_\zeta f}\right)\cot{\left( \frac{\pi}{T}(\zeta-\zeta_1)\right)}\, d\zeta\\
    &= -\int_{\partial \confD_\rho} U\cot{\left( \frac{\pi}{T}(\zeta-\zeta_1)\right)}\, d\zeta  - \int_{\mathpzc{T}_n\cup \mathpzc{B}_n}  \frac{\partial_\zeta w}{\partial_\zeta f} \cot{\left( \frac{\pi}{T}(\zeta-\zeta_1)\right)}\, d\zeta,
 \end{aligned}
\ee
where $\mathpzc{T}_n$ and $\mathpzc{B}_n$ are the top and bottom boundary of $\Omega_n$, respectively. Recall that 
\[
	\partial_\zeta w \longrightarrow \mp \frac{1}{2\pi} \sum_{j}\gamma_j, \quad \cot(\zeta-\zeta_1)\longrightarrow \mp i \qquad \textrm{as } \imagpart\zeta \to \infty.
\]
Taking absolute values in~\eqref{first wave speed bound} and sending $n$ to infinity, we arrive at the bound 
\begin{align}\label{quickusepropositionbounds}
    |c| \leq \frac{M}{|T|}\left( \rho N_{\text{v}}\left\|\cot{\left( \frac{\pi}{T}(\placeholder-\zeta_1)\right)}\right\|_{L^{\infty}(\partial \mathcal{D}_{\rho})} + N_{\text{c}}\max_{k} |\gamma_k|\right).
\end{align}
On the other hand, the definition of $\mathcal{U}$ \eqref{defofmathcalU} ensures that $|T|>2\rho$ and $\min_{j\neq k}\{|\zeta_j - \zeta_k| \}>2\rho$. This combined with the singular behavior of cotangent \eqref{cotangentseries} furnishes the estimate 
\[
\left\|\cot{\left( \frac{\pi}{T}(\placeholder-\zeta_1)\right)}\right\|_{L^{\infty}(\partial \mathcal{D}_{\rho})} \lesssim M\max\left\{ \frac{1}{\rho} , |T|\right\}
\]
where the implied constant is universal. This combined with \eqref{quickusepropositionbounds} and \eqref{circulationbound} yields the claimed bound in~\eqref{circulationwavespeedbounds}.
\end{proof}

\subsection{Symmetries and their consequences} 

We will call a set $E \subset \mathbb{C}$ \textit{odd} if $-E = E$ and \textit{imaginary symmetric} if it is invariant under even reflection with respect to the imaginary axis and \emph{real symmetric} if is invariant under even reflection over the real axis. Finally, $E$ is said to be \emph{symmetric} if it is both real and imaginary symmetric. Observe that if $E$ is real symmetric and $E$ is odd, then $E$ is symmetric. We call a function $f$ defined on a real symmetric domain $E$ that satisfies $f=f^*$ \textit{real on real}, where $f^* \colonequals \overline{f(\overline{\placeholder})}$ is the Schwarz conjugate of $f$.  

In order to construct solutions of the hollow vortex problem exhibiting these types of symmetries requires restricting the domain and codomain of $\F$ to certain subspaces. In particular, we define
\begin{equation}\label{subspaces}
    \begin{aligned}
       & C_{\textup{ii}}^{\ell+\alpha}(\mathbb{T})\colonequals \{ \varphi\in C^{\ell+\alpha} (\mathbb{T}) : i^m\widehat{\varphi}_m \in i\mathbb{R} \} \qquad C_{\textup{ri}}^{\ell+\alpha}(\mathbb{T}) \colonequals \{\varphi \in C^{\ell+\alpha}(\mathbb{T}) : i^m\widehat{\varphi}_{m} \in \mathbb{R} \}.
    \end{aligned}
\end{equation}
Their relevance will become clear momentarily. 

First consider the special case where there are evenly many vortices on the fundamental strip, with each vortex uniquely paired to another by odd reflection. For example, this occurs for the staggered von Kármán vortex street. 

\begin{lemma}\label{symmetrylemma2}
Suppose that $M=2m$ for some $m \geq 1$, and that 
\begin{equation}\label{symmetrylemma2centers}
    \zeta_{m+k}=-\zeta_k \quad \textrm{and} \quad  \gamma_{m+k}=-\gamma_k \qquad \text{ for } k=1,\ldots, m.
\end{equation}
If $|\rho| > 0$ and 
\begin{equation}\label{symmetrylemma2conditions}
    \mu_{m+k} =-\mu_{k}(-\placeholder), \quad \nu_{m+k}=-\nu_{k}(-\placeholder), \quad Q_{m+k} = Q_k \qquad  \text{ for all }  k = 1, \ldots, m,
\end{equation}
then $f$ and $w$ are odd, $\mathscr{D} $ is odd, and for all $k = 1, \ldots, m$,
\[
	\left\{
	\begin{aligned}
		\A_{m+k}(\mu,\nu,Q,\lambda; \rho) & = 0 \\
		\B_{m+k}(\mu,\nu,Q,\lambda; \rho) & = 0 
	\end{aligned}
	\right.
	\qquad
		 \textrm{if and only if} 
	\qquad
	\left\{
	\begin{aligned}
		\A_k(\mu,\nu,Q,\lambda; \rho) & = 0 \\
		\B_k(\mu,\nu,Q,\lambda; \rho) & = 0. 
	\end{aligned}
	\right.
\]
\end{lemma}
\begin{proof}

Notice that if both $f$ and $w$ are odd and the (local) kinematic~\eqref{Kinematiconditionlocal} and (local) Bernoulli conditions~\eqref{bernoulliconditionlocal1} are satisfied on the boundary of the $k$-th vortex, then they are also satisfied on the boundary of the $(m+k)$-th vortex since $\zeta_{m+k}=-\zeta_k$. So we will prove that if \eqref{symmetrylemma2conditions} holds, then $f$ and $w$ are odd. 

Let us begin with $w$. For $w^{0}$ we have: 
\begin{align*}
    w^{0}(\zeta)
    &=\sum_{k=1}^{m}\frac{\gamma_k}{2iT}\log{\left(\frac{\displaystyle \sin{\left(\frac{\pi}{T}(\zeta-\zeta_k)\right)}}{\displaystyle \sin{\left(\frac{\pi}{T}(\zeta+\zeta_k)\right)}}\right)}.
\end{align*}
Thus, $w^0 = - w^0(-\placeholder)$. On the other hand, 
\begin{align*}
     \mathcal{Z}^{\rho}[\nu](\zeta) 
    &=\sum_{k=1}^{m}\frac{1}{2Ti}\int_{\mathbb{T}}\nu_k(\sigma)\cot{\left(\frac{\pi}{T}(\rho\sigma -\zeta_k-\zeta) \right)}\rho \,d\sigma  -  \sum_{k=1}^{m}\frac{1}{2Ti}\int_{\mathbb{T}}\nu_{k}(-\sigma)\cot{\left(\frac{\pi}{T}(\rho\sigma +\zeta_{k}-\zeta) \right)}\rho \,d\sigma\\
    &=:\sum_{k=1}^{m}I_k(\zeta)+\sum_{k=1}^{m} J_{k}(\zeta).
\end{align*}
A straightforward computation reveals that $ I_k(-\zeta) = -J_k(\zeta)$, and hence $w$ is odd in $\zeta$. The argument for $f$ is identical, so we omit it.\end{proof}

A similar result holds for the case where there are evenly many vortices, with the paired vortices being even reflections of each over the real axis. The proof is quite similar to the previous lemma, and is therefore omitted.

\begin{lemma}\label{symmetrylemma3}
Suppose that $M=2m$ for some $n\geq 1$, and that  
\begin{equation}\label{conjugateconidition}
    \zeta_{m+k} = \overline{\zeta_k}, \text{ and }\quad  \gamma_{m+k} = -\gamma_k,\quad \text{for } k=1,\ldots , m.
\end{equation}
If $|\rho| > 0$, $c\in \mathbb{R}$, and 
\begin{equation}\label{conjugateconditionsonmunu}
  \mu_{m+k} = \mu_k^*\,,\quad \nu_{m+k} = \nu_k^*,\quad Q_{m+k} = Q_k \quad \text{for all }\, k=1,\ldots, M
\end{equation}
then $f,w$ are real on real, $\overline{\mathscr{D}}=\mathscr{D}$ and for all $k=1,\ldots , M$,
\[
	\left\{
	\begin{aligned}
		\A_{m+k}(\mu,\nu,Q,\lambda; \rho) & = 0 \\
		\B_{m+k}(\mu,\nu,Q,\lambda; \rho) & = 0 
	\end{aligned}
	\right.
	\qquad
		 \textrm{if and only if} 
	\qquad
	\left\{
	\begin{aligned}
		\A_k(\mu,\nu,Q,\lambda; \rho) & = 0 \\
		\B_k(\mu,\nu,Q,\lambda; \rho) & = 0. 
	\end{aligned}
	\right.
\]
\end{lemma}
\begin{remark}\label{symmetryremark?}
Consider the special situation that \eqref{conjugateconidition} holds (for $c\in \mathbb{R}$) and all the vortex centers lie on the imaginary axis. It is immediate that $\zeta_{m+k}=\overline{\zeta_k} = -\zeta_k$, thus \eqref{symmetrylemma2centers} holds as well. Assume now that both ~\eqref{symmetrylemma2conditions} and~\eqref{conjugateconditionsonmunu} are true, which amounts to requiring that 
\[
\mu_{m+k}(\tau) = -\mu_{k}(-\tau) = \overline{\mu_k(\overline{\tau})}, \quad \nu_{m+k}(\tau) = -\nu_{k}(-\tau) = \overline{\nu_k(\overline{\tau})}\quad \text{ for all } k=1,\ldots ,m,
\]
or equivalently 
\begin{equation}\label{imagimagdef}
    \mu_{k},\nu_k \in C_{\textup{ii}}^{\ell+\alpha}(\mathbb{T})
\end{equation}
In this case, according to Lemmas~\ref{symmetrylemma2} and \ref{symmetrylemma3}, we have that $f^* = f = -f(-\placeholder)$ and $w^*=w = -w(-\placeholder)$. 
\end{remark}
The preceding remark has the following, rather important, consequence. 
\begin{lemma}\label{symmetrylemma4}
Suppose that $\zeta_k \in i\mathbb{R}$ for all $k=1,\ldots,M$, the hypotheses of Lemma~\ref{symmetrylemma3} hold and $\mu,\nu\in C_{\textup{ii}}^{\ell+\alpha}(\mathbb{T})^M $. Then $\mathscr{D}$ is symmetric, and
\be
	\label{symmetries of A and B array}
\mathscr{A}_{k}^* (\tau ) = -\mathscr{A}_{k}(-\overline{\tau}), \quad \mathscr{B}_k^*(\tau)=\mathscr{B}(-\tau),\quad \text{for all } k=1,\ldots ,M. 
\ee
Accordingly, $\range{\mathscr{F}_k} \subset C_{\textup{ii}}^{\ell-1 + \alpha} (\mathbb{T}) \times C_{\textup{ri}}^{\ell-1+\alpha}(\mathbb{T})$. 
\end{lemma}
\begin{proof}
By Remark~\ref{symmetryremark?}, the assumption $\mu,\nu \in C^{\ell+\alpha}_{\textup{ii}}(\mathbb{T})$, and Lemmas~\ref{symmetrylemma2} and \ref{symmetrylemma3} it follows that $f$ and $w$ are odd, real on real, and thus $\mathscr{D}$ is symmetric. Proving~\eqref{symmetries of A and B array} is best done through the local formulation. Recall that according to \eqref{definitionsforscrA,B} and \eqref{Kinematiconditionlocal},
\begin{align*}
    \mathscr{A}_k^*(\tau)=\mathscr{A}_k (\overline{\tau}) &= \realpart{ \left(\overline{\tau}\left( \partial_\zeta w (\zeta_k+\rho\overline{\tau}) - c\partial_\zeta f(\zeta_k + \rho\overline{\tau}) \right) \right)}.
\end{align*}
The symmetries of $f,w$ easily lead to 
\begin{equation}\label{symmetriesoff,wproof}
    \partial_\zeta f (\zeta_k + \rho \overline{\tau}) = \overline{\partial_\zeta f (\zeta_k-\rho\tau)}, \quad \partial_\zeta w (\zeta_k + \rho \overline{\tau}) = \overline{\partial_\zeta w (\zeta_k-\rho\tau)}
\end{equation}
for all $k=1,\ldots, M$. Thus it directly follows that $\mathscr{A}_k^* = -\mathscr{A}_k(-\,\cdot)$. The result for $\mathscr{B}$ is proved similarly using the local Bernoulli condition \eqref{bernoulliconditionlocal1} and the definition \eqref{definitionsforscrA,B}.
\end{proof}
Now we prove that the situation is similar for a $2$P configuration. 
\begin{lemma}\label{2Pfirstlemma}
Suppose that $M=4$ and
\begin{equation}\label{2Psetup}
    \zeta_1\in i\mathbb{R}, \quad \zeta_2 = \frac{T}{2} + i\mathbb{R}, \quad \zeta_3 = \overline{\zeta_1},\quad \zeta_4 = \overline{\zeta_2}, \quad \gamma_3 = -\gamma_1 , \quad \gamma_4 = -\gamma_2.
\end{equation}
If $\rho > 0$, $c\in \mathbb{R}$, and \eqref{conjugateconditionsonmunu} holds, then $f$ and $w$ are both real on real and odd, $\mathscr{D}$ is symmetric, and
\[
\mathscr{A}_{k}^* (\tau ) = -\mathscr{A}_{k}(-\overline{\tau}), \quad \mathscr{B}_k^*(\tau)=\mathscr{B}(-\tau),\quad \text{for all } k=1,\ldots ,M. 
\]
Accordingly, $\range{\mathscr{F}_k} \subset C_{\textup{ii}}^{\ell-1 + \alpha} (\mathbb{T}) \times C_{\textup{ri}}^{\ell-1+\alpha}(\mathbb{T})$. 
\end{lemma}
\begin{proof}
The proof follows that of Lemmas~\ref{symmetrylemma2} and~\ref{symmetrylemma3} so we only highlight the main differences.  The elementary fact $\cot{(\pi/2 + \placeholder)} = -\tan$ easily leads to oddness of $w^0$. Thus, it suffices to show that $\mathcal{Z}^{\rho}\mu$ is odd. Observe that \eqref{conjugateconditionsonmunu} and $\nu_1,\nu_2 \in  C_{\textup{ii}}^{\ell+\alpha}(\mathbb{T})$ imply that $\nu_3(\tau) =-\nu_1(-\tau) $ and $\nu_4(\tau) = -\nu_2(-\tau)$ according to Remark \ref{symmetryremark?}. In the notation of the proof of Lemma \ref{symmetrylemma2} it suffices to show that $I_1(-\zeta) =- J_1(\zeta)$ and $I_2(-\zeta) = -J_2(\zeta)$; the former equality has the same proof as in Lemma \ref{symmetrylemma2}. For the latter we have
\begin{align*}
     I_{2} (-\zeta) &= \frac{1}{2Ti}\int_{\mathbb{T}}-\nu_{2}(\sigma)\cot{\left( \frac{\pi}{T}(\rho (-\sigma) - T/2-i\imagpart\zeta_2-\zeta)\right)}\rho \, d\sigma & \\
       &=\frac{1}{2Ti}\int_{\mathbb{T}}-\nu_{2}(\sigma)\cot{\left( \frac{\pi}{T}(\rho (-\sigma) + T/2+i\imagpart\zeta_4-\zeta)\right)}\rho \, d\sigma  = -J_{2}(\zeta).
\end{align*}
Thus, $f$ and $w$ are real on real and odd. As in the proof of Lemma~\ref{symmetrylemma4}, we easily verify that \eqref{symmetriesoff,wproof} still holds for this configuration and the result follows. 
\end{proof}

\subsection{Desingularization of symmetric configurations}

First consider the translating configuration of a von K\'arm\'an vortex street, which exhibits the symmetry of Lemma~\ref{symmetrylemma2},
\[
\Lambda_{\textup{vk}}^{0} \colonequals \left(\zeta_1=\tfrac{\pi}{4}+i,~\zeta_2=-\tfrac{\pi}{4}-i,~\gamma_1=1,~\gamma_2=-1,~ c= \tfrac{\tanh{(2)}}{2\pi},~ T= \pi \right).
\]
\begin{corollary}[von Kármán vortex street] \label{VonKarmanCorollaryfinal}
Let $\Lambda_{\textup{vk}}^{0}$ be given as above. 
There exists a global curve $\mathscr{C}_{\textup{vk}}$ of solutions to the periodic steady hollow vortex problem that admits the $C^0$ parametrization
\[
	\cm_{\textup{vk}}=\{ (\mu(s),\nu(s),Q(s),c(s),\rho(s) ): s\in [0,\infty) \}
\]
with all of the remaining parameters fixed to their values in $\Lambda_{\textup{vk}}^{0}$. The curve exhibits the following properties. 
\begin{enumerate}[label=\rm(\alph*)]
	\item\textup{(Bifurcation point)} \label{VKfinalbifurcation} The curve bifurcates from the von Kármán point vortex street in that
	\[
		\mu(0) = 0, \qquad \nu(0) = 0, \qquad Q(0) =0,\qquad c(0)=\tfrac{\tanh{2}}{2\pi},\qquad \rho(0)=0.
	\]
    which corresponds to \eqref{VKintrocoordinates}.
	\item\textup{(Blowup)} \label{Vkfinalblowup}  For each $s>0$ we have $\rho(s) > 0$, and in the limit 
      \begin{equation}
      \lim_{s\to \infty}\left( N_{\textup{c}} (s) + N_{\textup{v}}(s)\right) =\infty.
      \end{equation}

      \item\textup{(Symmetry)} \label{VKfinalsymmetry} The densities $\mu(s),\nu(s)\in \mathring{C}^{\ell+\alpha}(\mathbb{T})^2$ and normalized Bernoulli constant $Q(s) \in \mathbb{R}^2$ satisfy,
      \begin{equation}
      \mu_2(s) = -\mu_1(s)(-\placeholder), \quad \nu_2(s) = -\nu_1(s)(-\placeholder), \quad \text{ and } \quad Q_2(s) = Q_1(s).
      \end{equation}
	Accordingly, the fluid domain $\mathscr{D}(s) \colonequals  f(s)(\confD(s))$ exhibits the symmetry $-\mathscr{D}(s) = \mathscr{D}(s)$ and the relative velocity field is even. 
      \end{enumerate}
\end{corollary}
\begin{proof}
We start by proving the existence of a local curve of solutions $\mathscr{C}_{\textup{vk,loc}}$. Slightly abusing notation, we consider the new unknown $u= (\mu_1,\nu_1,Q_1,c)$ and put
\begin{equation}\label{BOGOVKchoice}
\mu_2=-\mu_1(-\placeholder),\quad \nu_2=-\nu_1(-\placeholder),\quad Q_2=Q_1.
\end{equation}
Thanks to Lemma~\ref{symmetrylemma2}, the symmetries of $\Lambda_{\textup{vk}}^{0}$ and \eqref{BOGOVKchoice} we have that the kinematic and Bernoulli conditions are satisfied on the first vortex boundary if and only if they are satisfied on the second vortex boundary. This justifies defining the new nonlinear operator by
\begin{equation}\label{BOGOVK}
    \mathscr{F}_{\textup{vk}}(\mu_1,\nu_1,Q_1,c ;\rho) \colonequals \F_1(\mu_1, -\mu_1(-\placeholder), \nu_1, -\nu_1(-\placeholder), Q_1, Q_1,c ; \rho),
\end{equation}
with domain and range,
\[
\mathscr{F}_{\textup{vk}} \colon \mathcal{O}_{\textup{vk}}\subset \mathring{C}^{\ell+\alpha}(\mathbb{T})\times \mathring{C}^{\ell+\alpha}(\mathbb{T})\times \mathbb{R}\times \mathbb{C}\times \mathbb{R}\rightarrow \mathring{C}^{\ell-1+\alpha}(\mathbb{T})\times C^{\ell-1+\alpha}(\mathbb{T}).
\]
The open set  $\mathcal{O}_{\textup{vk}}$ is derived from $\mathcal{O}$ in \eqref{definition for calO} the obvious way.  The analysis is much the same as in the proof Theorem \ref{localtheorem}; bijectivity of the linearized operator at the point vortex configuration is equivalent to the bijectivity of the row-reduced operator 
\[
\mathscr{L}_{\textup{vk}} \colon \mathring{C}^{\ell+\alpha}(\mathbb{T})\times\mathbb{R}^3\rightarrow C^{\ell-1+\alpha}(\mathbb{T}), \quad \mathscr{L}_{\textup{vk}}\begin{bmatrix}
    \dot{\mu}_1\\ \dot{Q}_1\\ \dot{c}
\end{bmatrix} \colonequals -\frac{1}{2\pi^2}\realpart \mathcal{C}^0 \dot{\mu}_1'-\dot{Q}_1-\frac{2}{\pi}\realpart(i\tau \dot{c}).
\]
It is apparent that $\mathscr{L}_{\textup{vk}}$ is an isomorphism, which furnishes the existence of the local curve $\mathscr{C}_{\text{vk,loc}}$ by the implicit function theorem. An application of Theorem~\ref{prelim global theorem} extends $\mathscr{C}_{\textup{vk,loc}}$ to the global curve $\mathscr{C}_{\textup{vk}}$. Observe that the analysis of Section~\ref{global section} holds for $\mathscr{F}_{\textup{vk}}$ because, by Lemma~\ref{symmetrylemma2} and \eqref{BOGOVKchoice}, every element of its zero-set uniquely corresponds to an element of $\F^{-1}(0)$.  All of the a priori estimates are therefore still valid by the same arguments. The statement in part~\ref{VKfinalbifurcation} is immediate; part~\ref{VKfinalsymmetry} is likewise an easy consequence of~\eqref{BOGOVKchoice} and Lemma~\ref{symmetrylemma2}. 

Since both circulations are fixed, by Theorem~\ref{global desingularization theorem}, to prove the statement in~\ref{Vkfinalblowup}, we need only show that $\lambda(s)=c(s)$ can be controlled in terms of $\Nconf(s)$ and $\Nvel(s)$. In particular, the inequality~\eqref{goodeinequality} gives a uniform lower bound on $\rho(s)$ in terms of $\Nconf(s)$. Then, applying the wave speed bound in~\eqref{circulationwavespeedbounds} completes the proof. \end{proof}

Now consider the periodic translating configuration of an array of point vortices:
\begin{equation*}
    \Lambda_{\mathrm{ar}}^0 \colonequals \left( \zeta_1=i, \, \zeta_2=-i, \, \gamma_1=1,\, \gamma_2=-1,\, c=\tfrac{\coth(2)}{2\pi}, \, T = \pi \right).
\end{equation*}
\begin{corollary}[Vortex array] \label{ArrayCorollaryfinal}
Let $\Lambda_{\textup{ar}}^{0}$ be given as above. 
There exists a global curve $\mathscr{C}_{\textup{ar}}$ of solutions to the periodic steady hollow vortex problem that admits the $C^0$ parametrization
\[
	\cm_{\textup{ar}}=\{ (\mu(s),\nu(s),Q(s),c(s),\rho(s) ): s\in [0,\infty) \}
\]
with all of the remaining parameters fixed to their values in $\Lambda_{\textup{ar}}^{0}$. The curve exhibits the following properties.  
\begin{enumerate}[label=\rm(\alph*)]
	\item\textup{(Bifurcation point)} \label{ARfinalbifurcation} The curve bifurcates from the (unstaggered) point vortex array in that
	\[
		\mu(0) = 0, \qquad \nu(0) = 0, \qquad Q(0) =0,\qquad c(0)=\tfrac{\coth{2}}{2\pi},\qquad \rho(0)=0.
	\]
    which corresponds to \eqref{ARintrocoordinates}.
	\item\textup{(Blowup)} \label{ARfinalblowupC}  For each $s>0$ we have $\rho(s) > 0$, $c(s)\in \mathbb{R}$ and in the limit 
      \begin{equation}
      \lim_{s\to \infty}\left( N_{\textup{c}} (s) + N_{\textup{v}}(s)\right) =\infty.
      \end{equation}
      \item\textup{(Symmetry)} \label{ARfinalsymmetry} The densities $\mu(s),\nu(s)\in \mathring{C}_{\textup{ii}}^{\ell+\alpha}(\mathbb{T})^2$ and normalized Bernoulli constant $Q(s) \in \mathbb{R}^2$ satisfy
      \begin{equation}
      \mu_2(s)  = \mu_1(s)^*, \quad \nu_2(s) = \nu_1(s)^*, \quad \text{ and } \quad Q_2(s) = Q_1(s).
      \end{equation}
	Accordingly, the fluid domain $\mathscr{D}(s)\colonequals f(s)(\confD(s))$ is symmetric and the relative velocity field is real on real and even. 
      \end{enumerate}
\end{corollary}
\begin{proof}
Similar to the proof of Corollary~\ref{VonKarmanCorollaryfinal}, we consider the unknown $u=(\mu_1,\nu_1,Q_1,c)$ and put
\[
\mu_2 = \mu_1^{*},\quad \nu_2=\nu_1^{*}, \quad Q_2=Q_1.
\]
It suffices to work with the nonlinear operator, 
\begin{equation*}
    \mathscr{F}_{\textup{ar}}(\mu_1,\nu_1,Q_1,c ;\rho) \colonequals \F_1(\mu_1, \mu_1^*, \nu_1, \nu_1^*, Q_1, Q_1,c ; \rho), 
\end{equation*}
thanks to Lemma~\ref{symmetrylemma3}. However in this case according to Lemma~\ref{symmetrylemma4} further restrictions can and should be made on the domain and codomain of $\mathscr{F}_{\textup{ar}}$. Namely, we consider
\[
\mathscr{F}_{\textup{ar}} \colon \mathcal{O}_{\textup{ar}}\subset 
\mathring{C}_{\textup{ii}}^{\ell+\alpha}(\mathbb{T})\times\mathring{C}_{\textup{ii}}^{\ell+\alpha}(\mathbb{T})\times \mathbb{R}^3\rightarrow \mathring{C}_{\textup{ii}}^{\ell-1+\alpha}(\mathbb{T})
\times C_{\textup{ri}}^{\ell-1+\alpha}(\mathbb{T}),
\]
for an open set $\mathcal{O}_{\textup{ar}}$ defined in the obvious way. The fact that $\mathscr{F}_{\textup{ar}}$ is well defined is a direct application of Lemma~\ref{symmetrylemma4}. We compute the linearized operator at the point vortex configuration exactly as in proof of Lemma~\ref{kernelrangelemma}. Thanks to the multiplier formulas~\eqref{operators}, we see that $\mathscr{A}_{1\nu_1} \colon \mathring{C}^{\ell + \alpha}_{\textup{ii}}(\mathbb{T})\to \mathring{C}^{\ell + \alpha}_{\textup{ii}}(\mathbb{T})$ is an isomorphism, so it suffices to consider the row-reduced operator:
\begin{equation}\label{rowredtranslating}
\mathscr{L}_{\textup{ar}}\begin{bmatrix}
    \dot{\mu_1}\\ \dot{Q}_1\\ \dot{c}
\end{bmatrix}:=-\frac{1}{2\pi^2}\realpart{\mathcal{C}^{0}\dot{\mu}_1^\prime}-\dot{Q}_1-\frac{2}{\pi}\realpart{(i\tau \dot{c})}
\end{equation}
now viewed as a mapping $\mathring{C}_{\imagimag}^{\ell+\alpha}\times \mathbb{R}\times \mathbb{R}\rightarrow C_{\textup{ri}}^{\ell-1+\alpha}(\mathbb{T})$. But again resorting to the multiplier formulas \eqref{operators} we have, 
\[
\range \realpart(\mathcal{C}^{0}\partial_{\tau})\big|_{\mathring{C}_{\imagimag}^{\ell+\alpha}}=P_{>1}C_{\realimag}^{\ell-1+\alpha}(\mathbb{T})
,\]
and $P_{\leq 1}C_{\textup{ri}}^{\ell-1+a}$ is a subset of the range $\mathbb{R}\times \mathbb{R}\ni(\dot{Q_1},\dot{c})\mapsto -\dot{Q_1}-2/\pi\realpart (i\tau \dot{c})$. Thus $\mathscr{L}_{\textup{ar}}$ is bijective. The implicit function theorem then furnishes the existence of a local curve $\mathscr{C}_{\textup{vk,loc}}$, which can be extended globally to the curve $\cm_{\textup{ar}}$ exactly as in the proof of Corollary~\ref{VonKarmanCorollaryfinal}. 
\end{proof}

Finally, consider the following $2$P configuration: 
\begin{align*}
      \Lambda_{\textup{2P}}^0 & \colonequals \Big( \zeta_1=i, \, \zeta_2=2i + \tfrac{\pi}{2},\, \zeta_3=-i,\,\zeta_4 = -2i + \tfrac{\pi}{2},\, \gamma_1 = \gamma_1^0, \\
      &\qquad\gamma_2=\gamma_2^0 ,\,\gamma_3 = -\gamma_1^0,\, \gamma_4=-\gamma_2^0,\, c=1,\,T=\pi\Big),
 \end{align*}
for $\gamma_1^0,\,\gamma_2^0$ as in~\eqref{CUT1}.

\begin{corollary}[2P vortices] \label{2Pcorolaryfinal}
Let $\Lambda_{\textup{2P}}^{0}$ be given as above. 
There exists a global curve $\mathscr{C}_{\textup{2P}}$ of solutions to the periodic steady hollow vortex problem that admits the $C^0$ parametrization
\[
	\cm_{\textup{2P}}=\{ (\mu(s),\nu(s),Q(s),c(s),\gamma_2(s),\gamma_4(s), \rho(s) ): s\in [0,\infty) \}
\]
with all of the remaining parameters fixed to their values in $\Lambda_{\textup{2P}}^{0}$. The curve exhibits the following properties.
\begin{enumerate}[label=\rm(\alph*)]
	\item\textup{(Bifurcation point)} \label{2pfinalbifurcation} The curve bifurcates from the $2$P point vortex configuration in that
    \begin{align*}
        \mu(0) = 0, \quad \nu(0) = 0, \quad Q(0) =0,\quad c(0)=1, \quad \gamma_2(0)=- \gamma_4(0) = \gamma_2^0,\quad \rho(0)=0.
    \end{align*}
	\item\textup{(Blowup)} \label{2PfinalblowupC}  For each $s>0$ we have $\rho(s) > 0$, $c(s)\in \mathbb{R}$, and in the limit 
      \begin{equation}
      \lim_{s\to \infty}\left( N_{\textup{c}} (s) + N_{\textup{v}}(s)\right) =\infty.
      \end{equation}
      \item\textup{(Symmetry)} \label{2Pfinalsymmetry} The densities $\mu(s),\nu(s)\in \mathring{C}_{\textup{ii}}^{\ell+\alpha}(\mathbb{T})^4$, circulations $\gamma_2(s), \gamma_4(s)$, and Bernoulli constant $Q(s)$ satisfy
      \begin{align*}
     &      \mu_3(s)  = \mu_1(s)^*, \quad \nu_3(s) = \nu_1(s)^*,  \quad Q_3(s) = Q_1(s), \\
     &       \mu_4(s)  = \mu_2(s)^*, \quad \nu_4(s) = \nu_2(s)^*, \quad  Q_4(s) = Q_2(s), \quad \text{ and } \quad \gamma_4(s) = -\gamma_2(s). 
      \end{align*}
	Accordingly, the fluid domain $\mathscr{D}(s)\colonequals f(s)(\confD(s))$ is symmetric and the relative velocity field is real on real and even. 
      \end{enumerate}
\end{corollary}
\begin{proof}
Similar to the proof of Corollary~\ref{VonKarmanCorollaryfinal}, we consider the unknown $u=(\mu_1,\mu_2,\nu_1,\nu_2,Q_1,Q_2,c,\gamma_2)$. For the other two vortices we put
\[
\mu_3=\mu_1^*,\quad \mu_4=\mu_2^*,\quad \nu_3=\nu_1^*,\quad \nu_4=\nu_2^*,\quad Q_3=Q_1,\quad Q_4=Q_2,\quad \gamma_4 = -\gamma_2
.\]
According to Lemma~\ref{symmetrylemma2}, this choice ensures that the boundary conditions are satisfied on the third and fourth vortex boundaries if and only if they are satisfied on the first two vortex boundaries. Define a new nonlinear operator $\mathscr{F}_{\textup{2P}} = (\mathscr{F}_{\textup{2P},1},\mathscr{F}_{\textup{2P},2})$ via
\begin{align*}
     \mathscr{F}_{\textup{2P},j}(u;\rho)&\colonequals  \mathscr{F}_j(\mu_1,\mu_2,\mu_1^*,\mu_2^*,\nu_1,\nu_2,\nu_1^*,\nu_2^*,Q_1,Q_2,Q_1,Q_2,\gamma_2,-\gamma_2,c;\rho)
\end{align*}
for $j = 1,2$. By Lemma~\ref{2Pfirstlemma}, these are well-defined as a mappings
\[
\mathscr{F}_{\textup{2P},j} \colon \mathcal{O}_{\textup{2P}}\subset \mathring{C}_{\textup{ii}}^{\ell+ \alpha}(\mathbb{T})^2\times\mathring{C}_{\textup{ii}}^{\ell+ \alpha}(\mathbb{T})^2\times\mathbb{R}^2\times \mathbb{R}^2\times \mathbb{R}\rightarrow \mathring{C}_{\textup{ii}}^{\ell-1+\alpha}(\mathbb{T})\times C_{\textup{ir}}^{\ell-1+\alpha}(\mathbb{T}), 
\]
where $\mathcal{O}_{\textup{2P}}$ is the open set derived from $\mathcal{O}$ in the obvious way. As in the proof of Corollary~\ref{VonKarmanCorollaryfinal}, we investigate bijectivity of the row-reduced operator $\mathscr{L}_{\textup{2P}}=(\mathscr{L}_{\textup{2P},1},\mathscr{L}_{\textup{2P},2})$. Explicit computation reveals that
\begin{align*}
    &\partial_{c}\mathcal{V}_1 (\Lambda_{\textup{2P}}^0) = \partial_{c}\mathcal{V}_2 (\Lambda_{\textup{2P}}^0)=-1,\qquad
    \partial_{\gamma_2}\mathcal{V}_1 (\Lambda_{\textup{2P}}^0) = \frac{\tanh{1} - \tanh{3}}{2\pi},
    \qquad \partial_{\gamma_2}\mathcal{V}_2 (\Lambda_{\textup{2P}}^0) = \frac{\coth4}{2\pi}.    
\end{align*}
Thus, the row-reduced operator takes the form 
\[
\mathscr{L}_{\textup{st},j}
\begin{bmatrix}
    \dot{\mu}\\ \dot{Q} \\ \dot{c} \\\dot{\gamma}_2
\end{bmatrix}=
-\frac{(\gamma_j^0)^2}{2\pi^2}\realpart \mathcal{C}^{0}\dot{\mu}_j'-\dot{Q}_j-\frac{2 \gamma_j^0}{\pi}\realpart{\left(i\tau D_{(c,\gamma_2)}\mathcal{V}_j(\dot{c},\dot{\gamma_2}) \right)}
.\]
Observe that $D_{(c,\gamma_2)}\mathcal{V}_j(\dot{c},\dot{\gamma_2})$ is  real valued, and bijective onto $\mathbb{R}^2$. Arguing as in Corollary \ref{ArrayCorollaryfinal} we see that $\mathscr{L}$ is an isomorphism. The rest of the proof follows along the same lines as Corollary~\ref{VonKarmanCorollaryfinal} where now we also make use of the circulation bound from Proposition~\ref{wavespeedProp}.
\end{proof}

\section*{Acknowledgments}

The work of both authors was supported in part by the NSF through DMS-2306243 and the Simons Foundation through award 960210.

\bibliographystyle{siam}

\bibliography{projectdescription}

\end{document}